\documentclass[12pt]{amsart}
\usepackage{amsthm,amssymb,amsmath,amscd,epic,eepic}
\usepackage[all]{xy}
\usepackage{graphicx}
\usepackage{enumerate}
\usepackage{epstopdf}
\usepackage{a4wide}
\usepackage[pagewise]{lineno}

\newtheorem{theorem}{Theorem}[section]
\newtheorem{lemma}[theorem]{Lemma}
\newtheorem{proposition}[theorem]{Proposition}

\newtheorem{corollary}[theorem]{Corollary}
\newtheorem{conjecture}[theorem]{Conjecture}
\theoremstyle{remark}
\newtheorem{remark}[theorem]{Remark}
\newtheorem{example}[theorem]{Example}

\def\S{\mathbin{\mathbb S}}
\def\T{\mathbin{\mathbb T}}
\def\R{\mathbin{\mathbb R}}
\def\Q{\mathbin{\mathbb Q}}
\def\Z{\mathbin{\mathbb Z}}
\def\C{\mathbin{\mathbb C}}
\def\P{\mathbin{\mathbb P}}
\newcommand{\bb}[1]{\mathbb{#1}}
\newcommand{\evj}{\eta_{\varepsilon_j}}
\newcommand{\ii}{\textrm{Int }}
\def\eor{\unskip\ \hglue0mm\hfill$\diamond$\smallskip\goodbreak}

\begin{document}
\markboth{Aleksandra Marinkovi\'c, Milena Pabiniak}
{Displaceability of pre-Lagrangians in contact toric manifolds.}

\title[On displaceability in contact toric manifolds]{ON DISPLACEABILITY OF PRE-LAGRANGIAN TORIC FIBERS IN CONTACT TORIC MANIFOLDS}

\author{A. MARINKOVI\'C, M. PABINIAK}
\address{CAMGSD, 
Mathematics Department, 
Instituto Superior T\'ecnico, Universidade de Lisboa\\
Av. Rovisco Pais, 1049-001 Lisboa, Portugal\\
aleksperisic@yahoo.com}
\address{Mathematisches Institut, Universit\"at zu K\"oln, Weyertal 86-90, D-50931 K\"oln, Germany
\\pabiniak@math.uni-koeln.de}

\maketitle

\begin{abstract}
In this note we analyze displaceability of 
pre-Lagrangian toric fibers
in contact toric manifolds. While every symplectic toric manifold contains at least one non-displaceable Lagrangian toric fiber and infinitely many displaceable ones, we show that this is not the case for contact toric manifolds. More precisely, we prove that for the contact toric manifolds $\S^{2d-1} (d\geq 2)$ and $\T^k \times \S^{2d+k-1} (d \geq 1)$ all pre-Lagrangian toric fibers are displaceable, and that for all contact toric manifolds for which the toric action is free, except possibly non-trivial $\T^3$-bundles over $\S^2$, all pre-Lagrangian toric fibers are non-displaceable. Moreover we also prove that if for a compact connected contact toric manifold all but finitely many pre-Lagrangian toric fibers are non-displaceable then the action is necessarily free. On the other hand, as we will discuss, displaceability of all pre-Lagrangian toric fibers seems to be related to the non-orderability of the underlying contact manifolds.
\end{abstract}
\section{Introduction}
One of the questions of great importance in symplectic geometry is whether a given Lagrangian submanifold of a symplectic manifold can be displaced off itself via a Hamiltonian isotopy. If the symplectic manifold is toric, i.e. it can be equipped with an effective Hamiltonian action of a torus of dimension equal to half of the dimension of the manifold, then every generic toric orbit (i.e. of maximal dimension) is a Lagrangian submanifold. It is usually called a Lagrangian toric fiber as it is a fiber of the moment map. Displaceability properties of Lagrangian toric fibers in symplectic toric manifolds have been extensively studied. In particular the following two important results have been proved
\begin{itemize}
\item[(A)] any compact connected symplectic toric manifold contains a non-displaceable Lagrangian toric fiber (\cite{fooo:2010}, \cite{fooo:2011}, \cite{entpol}, \cite{GW}, \cite{marinkovicpabiniak}), whereas 
\item[(B)] any compact connected symplectic toric manifold has uncountably many displaceable Lagrangian toric fibers  (\cite{mcduff}, \cite{abreubormanmcduff}).
\end{itemize}
The goal of this paper is to study the analogue of the above questions in the setting of contact toric manifolds.

The question of displaceability of Lagrangians in symplectic manifolds can be translated to the contact setting in two different ways.
Given a {\it pre-Lagrangian} $L^d$ in a contact manifold $(V^{2d-1},\xi)$ one can ask if there exists a contact isotopy $\varphi_t$ of $(V^{2d-1},\xi)$ which displaces $L^d$.
Or, given a Legendrian submanifold $N^{d-1}$ of $(V^{2d-1},\xi)$ one can ask if there exists a contact isotopy $\varphi_t$ of $(V^{2d-1},\xi)$ such that there are no Reeb chords between $N^{d-1}$ and $\varphi_1(N^{d-1})$. 
In this work we concentrate on the first question as this is the direct translation of the problem of displaceability of generic torus orbits: 
if $(V^{2d-1},\xi)$ is a {\it contact toric manifold}, then the generic orbits of the toric action are pre-Lagrangian submanifolds and are fibers of the moment map (with the exception of $(\T^3, \ker \alpha_k)$, $k>1$, described in Section \ref{section contact manifolds}, where each fiber consists of $k$ orbits).
For a precise definition of pre-Lagrangians, moment maps in the contact setting and other background information we direct the reader to Section \ref{section contact manifolds}.
We show that displaceability properties of pre-Lagrangian toric fibers in contact toric manifolds are very different from those of Lagrangian toric fibers in the symplectic setting.
None of the above itemized statements holds when translated to the contact toric setting.

First we observe that contact toric manifolds may have all pre-Lagrangian toric fibers displaceable, i.e. the contact version of (A) does not hold.

\begin{proposition}\label{displace sphere} 
Every pre-Lagrangian toric fiber in the standard contact toric sphere $\mathbb{S}^{2d-1},$ $d\geq2$, is displaceable.
\end{proposition}
In fact all closed proper subset of $(\S^{2d-1},\xi_{st})$ are displaceable (Proposition \ref{displace all closed}). 
It is also worth observing that there exists a contact isotopy that simultaneously displaces all these pre-Lagrangian toric fibers (Remark \ref{displace all with one iso}). Moreover, the standard contact sphere is not the only instance of this phenomenon.
\begin{theorem}\label{displace in boundary of stabilization}
Every pre-Lagrangian toric fiber in $\T^k \times \S^{2d+k-1}$, $d\geq1$, is displaceable. 
\end{theorem}
This phenomenon seems to be related to the notion of {\it orderability} introduced by Eliashberg and Polterovich in \cite{elpo} and recalled here in Section \ref{section contact manifolds}. 
The contact manifolds $\S^{2d-1}$ and $\T^k \times \S^{2d+k-1} (d \geq 2)$ are known to be not orderable \cite{EKimP}. To the best of our knowledge, at the time of writing these are all contact toric manifolds which are {\it proved} to be non-orderable. Orderability of $\T^k \times \S^{k+1}$ remains an open question, even for $k=1$. The fact that Theorem \ref{displace in boundary of stabilization} holds also for $d=1$ provides a slight indication that $\T^k \times \S^{k+1}$ might also be non-orderable.

 A connection between orderability and non-displaceability of certain subsets of a contact manifold has already been studied. It was proved by Eliashberg and Polterovich that a contact manifold is orderable if it contains a pair with the stable intersection property (Theorem 2.3.A in \cite{elpo}),
that is a pair $(L,K)$, where $L$ is a pre-Lagrangian and $K$ is either a pre-Lagrangian or a Legendrian, such that the stabilizations of $K$ and $L$ cannot be displaced from each other in $V \times T^*\S^1$, the stabilization of $V$. (For precise definitions see Section 2.2 in \cite{elpo}.)
  If a pre-Lagrangian $L$ paired with itself has stable intersection property, i.e. if it is stably non-displaceable, then it is non-displaceable.  Moreover, Borman and Zapolsky in \cite{BormanZapolsky} proved that any compact contact toric manifold admitting a monotone quasimorphism with the vanishing property is orderable and contains a non-displaceable pre-Lagrangian toric fiber. 
 Theorems \ref{displace sphere} and \ref{displace in boundary of stabilization} imply that neither $\mathbb{S}^{2d-1},$ with $d\geq2$, nor $\T^k \times \S^{2d+k-1}$, with $d\geq 1$, can be equipped with a monotone quasimorphism with the vanishing property
 \footnote{ Here we would like to mention another result related to the non-existence of monotone quasimorphism for the standard sphere: Fraser, Polterovich and Rosen in \cite{FPR} proved that
  every conjugation invariant norm on the identity component of the universal cover of the contactomorphism group of $\mathbb{S}^{2d-1}$, $d\geq2$, must be bounded and discrete, hence equivalent to the trivial norm.}.
 Based on our work and the above results, we conjecture
\begin{conjecture} \label{conjecture}
 If a compact contact toric manifold is not orderable then all its pre-Lagrangian toric fibers are displaceable.
   \end{conjecture}
Note that the compactness assumption is crucial. 
The contact toric manifold $(\mathbb{S}^1\times\mathbb{R}^{2d},\ker (d\theta-\sum y_jdx_j))$ is orderable (\cite{Sandon}), even though all pre-Lagrangian toric fibers are displaceable (see Example \ref{example non-compact}). 
Moreover, $(\R^{2d+1}, \ker (dz-\sum y_jdx_j))$ is orderable (\cite{Bhupal}), while
every bounded subset of it is displaceable by a translation in the $z$ direction.

Another difference in rigidity properties of symplectic and contact manifolds is that the contact version of (B) does not hold: a contact toric manifold does not need to contain any displaceable pre-Lagrangian toric fiber.
  \begin{proposition}\label{non-displaceable cosphere} 
  Every pre-Lagrangian toric fiber in the contact toric manifold $\T^{d}\times\mathbb{S}^{d-1},$ $d\geq2$, (co-sphere bundle of $\T^d$), is non-displaceable.
\end{proposition}
This follows from results of Eliashberg, Hofer and Salamon from  \cite{EliashbergHoferSalamon}
where they used Lagrangian Floer homology in the symplectization of contact manifolds to analyze displaceability of graphs of non-vanishing $1$--forms in cosphere bundles, (see Section \ref{section free non-displaceable}). In fact, using Example 2.4.A of \cite{elpo} one can show a stronger result, namely that any pre-Lagrangian toric fiber in $\T^d\times\mathbb{S}^{d-1},d\geq2$, is stably non-displaceable. Due to Theorem 2.3.A of
 \cite{elpo} it follows that $\T^{d}\times\mathbb{S}^{d-1},d\geq2$, is orderable.
 
 \begin{proposition}\label{non-displaceable 3torus}
  Every orbit of the torus action in the contact toric manifold 
  $(\T^3=\S^1_{(\theta)}\times \T^2_{(\theta_1,\theta_2)}, \ker(\cos (2\pi k \theta)\,d\theta_1+\sin (2\pi k \theta)\,d\theta_2)),\,k\geq 1$, is non-displaceable. In particular, every pre-Lagrangian toric fiber of $\T^3$ with one of the above contact forms is non-displaceable.
 \end{proposition}
 
It is interesting to observe that for contact toric manifolds from Propositions \ref{non-displaceable cosphere} and \ref{non-displaceable 3torus} the toric action is free. We believe that the non-existence of displaceable pre-Lagrangian toric fibers is related to the 
fact that the given toric action is free. A toric action on a compact symplectic toric manifold is never free (it is free only on the pre-image of the interior of the moment map image). In fact even a Hamiltonian circle action on a compact symplectic toric manifold is never free: compactness implies that the Hamiltonian moment map must attain its extrema, and the non-degeneracy of the symplectic form implies that the Hamiltonian vector field vanishes at these points, thus these points must be fixed under the circle action. In the contact setting, the Hamiltonian vector field only needs to be in the kernel of $d\alpha$ at the points where the moment map attains its extrema.

All contact toric manifolds admitting a free toric action are listed in Lerman's Classification Theorem (Theorem 2.18 in \cite{Lerman}) which we recall in Section \ref{section contact manifolds}. The only contact toric manifolds with a free toric action that are not included in the statements of Propositions \ref{non-displaceable cosphere} and \ref{non-displaceable 3torus} are the non-trivial $\T^3$-bundles over $\mathbb{S}^2.$ It would be interesting to see if in these cases all pre-Lagrangian toric fibers are non-displaceable.
Finally, in Section \ref{subsection not free displace} we prove the following.
 \begin{theorem}\label{not free displaceable}
 Every compact connected contact toric manifold for which the toric action is not free contains uncountably many displaceable pre-Lagrangian toric fibers.
 \end{theorem}

{\bf Organization.} Section \ref{section contact manifolds} contains background material on contact manifolds and toric actions. In Section \ref{section methods for displacing} we present some methods of displacing pre-Lagrangian toric fibers by analyzing  prequantization maps, contact reduction and contact cutting. In Section \ref{displace not ord} we apply these tools to displace all pre-Lagrangian toric fibers of the non-orderable contact toric manifolds of Propositions \ref{displace sphere} and \ref{displace in boundary of stabilization}. Section \ref{section free non-displaceable} is devoted to study non-displaceability of pre-Lagrangian toric fibers of contact toric manifolds with free toric action. There we prove Propositions \ref{non-displaceable cosphere}, \ref{non-displaceable 3torus} and Theorem \ref{not free displaceable}. 

\section{Background on contact manifolds and torus actions.}\label{section contact manifolds}
\subsection{Contact manifolds and pre-Lagrangian submanifolds.}
Let $(V^{2d-1},\xi)$ be a cooriented contact manifold. Recall that the symplectization of $(V^{2d-1},\xi)$ is the symplectic manifold
$$SV=\{(p,\eta_p) \in T^*V \,| \hskip1mm 
\ker\,\eta_p=\xi_p\hskip1mm\textrm{and}\hskip1mm \eta_p \hskip1mm 
\textrm{agrees with the coorientation of}\hskip1mm \xi_p\}$$
 with the symplectic form $d\lambda_{|SV}$,  where $\lambda$ is the canonical Liouville 1-form on $T^*V.$
 The standard $\mathbb{R}_+$-action on $SV$ defined by $t\cdot(p,\eta_p)\mapsto(p,t\eta_p)$ makes $SV$ a principal $\mathbb{R}_+$-bundle over $V$.
 Any contact form for $\xi$ is a section of this bundle.
  We denote by $\pi:SV\rightarrow V$ the projection $\pi(p,\eta_p)=p.$
 A submanifold $L^d\subset (V^{2d-1},\xi)$ is said to be a \textbf{pre-Lagrangian} if it is the diffeomorphic image under $\pi$ of some Lagrangian submanifold $\widetilde{L}\subset SV.$ The notion of pre-Lagrangian submanifold and the related displaceability problems have been first studied in \cite{EliashbergHoferSalamon}.

We will analyze the problem of displaceability of pre-Lagrangians under contact isotopies. 
In symplectic topology every (possibly time-dependent) function $h_t : M\mapsto \R$ on a symplectic manifold $(M,\omega)$ induces a symplectic isotopy $\varphi_t$, which is defined to be the flow of the vector field $X_{h_t}$ determined by the relation $X_{h_t}\lrcorner\omega=dh_t.$ The isotopy $\{\varphi_t\}$ is said to be a Hamiltonian isotopy, with Hamiltonian function $h_t.$  A Lagrangian $L$ in a symplectic manifold $(M,\omega)$ is said to be non-displaceable if $\varphi_1(L)\cap L\neq\emptyset$ for every Hamiltonian isotopy 
 $\{\varphi_t\}_{t \in [0,1]}$ on $M$ with $\varphi_0=id$. In the contact setting, every (possibly time-dependent) function $h_t : V \mapsto \R$ on a contact manifold $(V, \xi = \ker \alpha)$ induces a contact isotopy ${\varphi_t}$, which is defined to be the flow of the vector field $X_{h_t}$ determined by the relation
 \begin{equation}\label{ham}
 \alpha(X_{h_t})=h_t\hskip2mm\textrm{and}\hskip2mm X_{h_t}\lrcorner d{\alpha}=dh_t(R_{\alpha})\alpha-dh_t.
 \end{equation}
 Here $R_{\alpha}$ denotes the Reeb vector field associated to a contact form $\alpha,$ that is, the unique vector field  such that
$$R_{\alpha}\lrcorner d\alpha=0 \hskip2mm\mathrm{and}\hskip2mm \alpha(R_{\alpha})=1.$$
 Note that $\{\varphi_t\}$ depends on the choice of a contact form $\alpha$. The function $h_t$ is then said to be the Hamiltonian function of the contact isotopy $\{\varphi_t\}$ with respect to the contact form $\alpha$. Contrary to the symplectic case, any contact isotopy is induced by a Hamiltonian function (which is defined uniquely by $h_t = \alpha(X_{h_t})$, where $X_t$ is the vector field generating the isotopy).
Translating the notion of non-displaceable Lagrangian fiber in a symplectic manifold to the contact setting we obtain the following definition. 
 A pre-Lagrangian $L\subset V$ is called \textbf{non-displaceable} if  $\varphi_1(L)\cap L\neq\emptyset$ for every contact isotopy 
 $\{\varphi_t\}_{t\in[0,1]}$ on $V$, with $\varphi_0=id$. Otherwise, it is called displaceable.
\subsection{Contact toric manifolds.}
A co-oriented contact manifold $(V^{2d-1},\xi)$ with an effective action of the torus $\T^d$ that preserves the contact structure $\xi$ is called a {\bf contact toric manifold}. 
As $\T^d$ is compact one can always choose a $\T^d$-invariant contact form $\alpha$ for $\xi$. For such a $\T^d$-invariant contact form $\alpha$, the \textbf{$\alpha$-moment map} $\mu_{\alpha} \colon V\rightarrow (\mathfrak{t}^d)^*$  is defined by
 $$\langle\mu_{\alpha}(p),X\rangle=\alpha_p(\underline{X}_p),$$
 where $\underline{X}$ is the vector field on $V$ generated by $X\in\mathfrak{t}^d$ and $\langle\cdot,\cdot\rangle$ denotes the natural pairing between $\mathfrak{t}^d$ and $(\mathfrak{t}^d)^*$. If $V$ is equipped with a contact toric action of $\T^d$, the lift of this action to $T^*V$ is symplectic and keeps $SV$ invariant, making $SV$ a symplectic toric manifold. The corresponding moment map $\Phi:SV\rightarrow(\mathfrak{t}^d)^*$, called the {\bf contact moment map}, is given by
 $$\langle\Phi(p,\eta_p),X\rangle=\eta_p(\underline{X}_p)$$
 and does not depend on the choice of a contact form. The {\bf moment cone} is the image of the contact moment map. 
One can identify the Lie algebra $\mathfrak{t}^d$ of $\T^d$ with $\mathbb{R}^d$ by fixing a splitting of $\T^d$ into a product of circles and an identification $Lie(\S^1)\cong \R$, and view the moment maps $\mu_{\alpha}$ and  $\Phi$ as maps to $(\bb{R}^d)^*$. 
We use the convention  $\S^1\cong\mathbb{R}/\mathbb{Z}$. 
Note that switching to a different convention would change $\mu_{\alpha}(V)$, but not the moment cone $\Phi(SV)$, as the cone is invariant under rescaling in the radial direction. Changing the splitting of $\T^d$, i.e. reparametrizing the action, would result in applying a $GL(d,\Z)$ transformation to $\mu_{\alpha}(V)$ and $\Phi(SV)$ (see Example \ref{example sphere is a prequantization}).
 
 For any $\T^d$-invariant contact form $\alpha$ it holds that $\Phi\circ\alpha=\mu_{\alpha}$ (Proposition 2.8 in \cite{Lerman}) and, for any $c \in \mu_{\alpha}(V) \cap Int \,\Phi(SV),$ 
 the restriction of the projection $\pi \colon SV \rightarrow V$ maps diffeomorphically
 $ \Phi^{-1}(c) $ onto $ \mu_{\alpha}^{-1}(c)$. Therefore generic fibers of $ \mu_{\alpha}$ are pre-Lagrangians. We call them the {\bf pre-Lagrangian toric fibers}. 
 Connected components of the fibers of $\mu_{\alpha}$ are $\T^d$-orbits (see Lemma 3.16 in \cite{Lerman}).
 If $\dim V>3$ then, by Theorem 4.2 in \cite{Lerman}, all fibers of $\mu_{\alpha}$ are connected. 
 In fact the only contact toric manifolds with disconnected pre-Lagrangian fibers are
 $(\T^3, \ker \alpha_k)$ described below, with $k>1$. (Each fiber has exactly $k$ connected components).
 
The classification of compact contact toric manifolds initiated by Banyaga-Molino \cite{BanyagaMolino1}, \cite{BanyagaMolino2} and Boyer-Galicki \cite{BoyerGalicki} was concluded by Lerman in \cite{Lerman}.
 Compact connected contact toric manifolds $(V, \xi)$ are classified as follows (Theorem 2.18 in \cite{Lerman}):
\begin{itemize}\label{classification}
\item Suppose $\dim V= 3$ and the torus action is free. Then $V$ is diffeomorphic to $\T^3=\S^1_{(\theta)}\times \T^2_{(\theta_1,\theta_2)}$ with the contact form $\alpha_k= \cos (kt)\,d\theta_1 +\sin (kt) \, d\theta_2$, for some $k \geq1$. The moment cone is $\R^2$. 
Note that $\T^3$ is the cosphere bundle of $\T^2$ and the standard contact structure of the cosphere bundle corresponds to $k=1$.
\item Suppose $\dim V= 3$ and the torus action is not free. Then $V$ is diffeomorphic to $\S^3,$ $\S^1 \times \S^2$ or a lens space. There are various possible toric actions and various possible contact structures (including overtwisted ones). For details see \cite{Lerman}.
\item Suppose $\dim V> 3$ and the torus action is free. Then $V$ is a principal $\T^{d}$-bundle over $\S^{d-1}$, and the moment cone is the whole $\R^{d}$.  Each such principal bundle has a unique $\T^d$-invariant contact structure making it a contact toric manifold.
Since principal $\T^d$-bundles over a manifold are in one-to-one correspondence
with the second cohomology classes of the manifold with coefficients in $\mathbb{Z}^d$ and since $H^2(\mathbb{S}^{d-1},\mathbb{Z}^d) = 0$
for $d-1\neq2,$ it follows that when $\dim V=2d-1>5,$ $V$ must be the trivial bundle $\T^d\times\mathbb{S}^{d-1}$.\item Suppose $\dim V>3$ and the torus action is not free. Then the contact toric manifold is uniquely determined by its convex moment cone,
up to $GL(2d-1,\Z)$ transformations (corresponding to changing a splitting of a torus into a product of circles). When the moment cone is strictly convex then $V$ is of Reeb type, i.e. $V$ admits a contact form whose Reeb vector field generates a circle subaction of the toric action. Otherwise, i.e. when the moment cone is convex but not strictly convex, $V$ is $\T^k\times\S^{2d+k-1}.$
\end{itemize}
 \subsection{Orderability.}  Eliashberg and Polterovich in \cite{elpo} defined a relation $\preceq$ on the universal cover of the identity component of the contactomorphism group, $\widetilde{Cont}_0(V,\xi)$:  two elements $[\{\phi_t\}], [\{\psi_t\}] \in \widetilde{Cont}_0(V,\xi)$ satisfy $[\{\phi_t\}] \preceq [\{\psi_t\}]$
if $[\{\psi_t\}] \circ [\{\phi_t\}]^{-1}$ can be represented by a non-negative contact isotopy, 
i.e. a contact isotopy that moves every point of $V$ in a direction 
positively transverse or tangent to $\xi$
(equivalently, a contact isotopy that is generated
by a non-negative contact Hamiltonian).
 This relation is always reflexive and transitive. If it is also anti-symmetric then it defines a bi-invariant partial order on $\widetilde{Cont}_0(V,\xi)$ and the contact manifold $(V,\xi) $ is called {\bf orderable}.
   Equivalently, according to Proposition 2.1.A in \cite{elpo}, a contact manifold is orderable if there are no contractible loops of contactomorphisms generated by a strictly positive contact Hamiltonian.
Eliashberg, Kim and Polterovich in Theorem 1.16 in \cite{EKimP} showed that
the ideal contact boundary of a product $M\times\mathbb{C}^d$ of a Liouville manifold $M$ and $(\mathbb{C}^d, \omega=\frac{i}{2}\sum dz\wedge d\bar{z})$ is not orderable for $d\geq 2$. 
Two important cases which we consider here are when $M$ is a point  and when $M=T^*\mathbb{S}^k$. In these cases the ideal contact boundary of $M$ stabilized $d$ times is, respectively, the standard contact spheres $\S^{2d-1}$ and the manifolds $\T^k \times \S^{2d+k-1}$ with the contact form which is described in detail in Section \ref{section displace by cuts}. Hence, $\S^{2d-1}$ and $\T^k \times \S^{2d+k-1}$ are non-orderable if $d\geq 2.$  In Section  \ref{displace not ord} we prove that all their pre-Lagrangian toric fibers are displaceable.
\section{Methods for displacing. }\label{section methods for displacing}
In this section we describe some methods of displacing pre-Lagrangian toric fibers in contact toric manifolds. All these methods are obtained by a similar pattern: we look at various ways of constructing a new manifold from a given one (prequantization, contact reduction, contact cut) and deduce relations between displaceability properties of corresponding subsets of the manifold we started with and of the manifold we constructed. 
\subsection{Prequantization of a symplectic toric manifold and displaceability}\label{section prequantization}
Let $(M,\omega)$ be a symplectic manifold such that the cohomology class $[\omega]$ is integral.
 The \textbf{prequantization} of $(M,\omega)$ is the principal $\mathbb{S}^1-$bundle  $\pi:V\rightarrow M$ with Euler class $[\omega]$. There is a connection 1-form $\alpha$ on $V$ such that
 $\pi^*\omega=d\alpha$, which is also a contact form on $V$, and thus $(V,\xi=\ker \alpha)$ is a contact manifold. This construction is due to Boothby-Wang \cite{BoothbyWang} (see also \cite[Section 7.2.]{Geiges}). 
 The orbits of the Reeb vector field $R_{\alpha}$ associated to the contact $1$-form $\alpha$ are the fibers of the bundle.
If $V$ is the prequantization of $(M,\omega)$ then for the subgroup $\mathbb{Z}_k=\{e^{2\pi i\,l/k};\,l=0,\ldots, k-1\}\subset\mathbb{S}^1$ the quotient submanifold $V/\mathbb{Z}_k$ is the prequantization of $(M,k\omega)$.

  If $(M^{2d},\omega)$ is a compact symplectic toric manifold then the $\T^d$-action lifts to a contact $\T^d$-action on $V$. Together with the $\S^1$-action given by the Reeb flow this gives a $\T^{d+1}$-action on the prequantization $(V^{2d+1},\xi)$ making it a contact toric manifold (see \cite{Lerman2}). The pre-image under $\pi$ of a Lagrangian toric fiber in $M$ is a pre-Lagrangian toric fiber in $V.$ Moreover, the image of the contact moment map $\Phi \colon SV \rightarrow (\bb{R}^d)^* \times (\bb{R})^*$ is a cone over the moment map image $\Delta$ of $M$, namely it is
  $$C=\{a(x,1) \in (\bb{R}^d)^* \times (\bb{R})^*\,|\, x \in \Delta, a \in \bb{R}_{> 0}\}.$$ Symplectic reduction of $SV$ with respect to the circle $\{1\} \times \S^1 \subset \T^{d+1}$ taken at level $1$ gives back the symplectic toric manifold $(M,\omega)$ (Lemma 3.7 in \cite{Lerman2}).

\begin{example}\label{example sphere is a prequantization} 
The standard contact sphere, $(\mathbb{S}^{2d-1},\ker \alpha_{st})$, $ \alpha_{st}=\frac{i}{4}\sum_{i=1}^d(z_id\bar{z}_i-\overline{z}_idz_i)$,
 is the prequantization of the complex projective space $(\mathbb{CP}^{d-1}, \frac{1}{\pi}\omega_{FS})$
  as the first Chern class of the Hopf fibration $\mathbb{S}^{2d-1}\rightarrow\mathbb{CP}^{d-1}$ is  $\frac{1}{\pi}\omega_{FS}$.
 Performing the above construction for the standard $\T^{d-1}$-action on $(\mathbb{CP}^{d-1},\omega_{FS})$, one equips the prequantization space  $\mathbb{S}^{2d-1}$ with a toric $\T^d$-action
 $$(t_1,\ldots,t_d)\ast (z_0,\ldots,z_{d-1})=(t_dz_0, t_dt_1z_1,\ldots, t_dt_{d-1} z_{d-1})$$
 (the Reeb flow gives the diagonal circle action).
 The corresponding moment cone is spanned by the directions $e_1+e_{d}, e_2+e_{d},\ldots, e_{d-1}+e_{d},e_{d},$ where $e_i,$ $i=1,\ldots d$ are the coordinate axes in $\R^d$.
Observe that the resulting action differs by a reparametrization from the standard action of $\T^d$ on $\mathbb{S}^{2d-1}$ induced from the
$\T^d$-action  on $\mathbb{C}^d$, where each circle in $\T^d$ rotates the corresponding copy of $\C$ with speed $1$
(see Section \ref{contact sphere}).
This is why the above moment cone differs by a $GL(d,\Z)$-transformation from the moment cone of the standard action of $\T^d$ on $\mathbb{S}^{2d-1}$ (which
is spanned by the directions $e_1,\ldots, e_d$; Section \ref{contact sphere}).
Furthermore,
   the real projective space $\mathbb{RP}^{2d-1}=\mathbb{S}^{2d-1}/\mathbb{Z}_2$, and more generally the lens spaces $L_p^{2d-1}=\mathbb{S}^{2d-1}/\mathbb{Z}_p,$ $ p\in\mathbb{N}$,
  with contact forms induced by $\alpha_{st}$ are prequantizations of $(\mathbb{CP}^{d-1},\frac{p}{\pi}\omega_{FS})$. Hence, they are also contact toric manifolds, 
  and their moment cones are spanned by the directions
  $pe_1+e_{d}, pe_2+e_{d},\ldots, pe_{d-1}+e_{d},e_{d}.$
 \end{example}
We now explain a connection between non-displaceability in a symplectic toric manifold $(M,\omega)$ and in its prequantization $(V,\xi)$. Note that if $L' \subset M$ is a Lagrangian submanifold then $\pi^{-1}(L')\subset V$ is a pre-Lagrangian submanifold. By $Ham(M,\omega)$ we denote the group of Hamiltonian diffeomorphisms on $(M,\omega)$ and by
$Cont_0(V,\xi)$ the identity component of the contactomorphisms group of $(V,\xi).$ 
 \begin{lemma}\label{lift}(\emph{Lifting property for prequantization}) Let $\varphi\in Ham(M,\omega).$ Then, there is $\widetilde{\varphi}\in Cont_0(V,\xi)$ such that
 $\pi\circ\widetilde{\varphi}=\varphi\circ\pi$.
 \end{lemma}
 \begin{proof} Let $\varphi_t$ be a Hamiltonian isotopy such that $\varphi=\varphi_1$ and let $h_t:M\rightarrow\mathbb{R}$ be the corresponding time dependent Hamiltonian function. Let $\widetilde{\varphi}_t$ be the contact isotopy generated by $\widetilde{h}_t=h_t\circ\pi$. Then $\widetilde{\varphi}=\widetilde{\varphi}_1$ has the desired properties. 

 \end{proof}
  What follows is an analogue of the result of Abreu and Macarini on preserving non-displaceability under symplectic reduction (Proposition 3.2 in \cite{abreumacarini}). The proof immediately follows from Lemma \ref{lift}.
  \begin{proposition}\label{lift.preq} Let $(V,\xi)$ be a contact toric manifold which is the prequantization of a symplectic toric manifold $(M,\omega).$ If a Lagrangian toric fiber $L\subset M$ is displaceable then so is the pre-Lagrangian toric fiber $\pi^{-1}(L)\subset V.$
 \end{proposition}
\begin{example}\label{example non-compact}
The contact toric manifold $(\S^1\times\R^{2d},\ker(d\theta-\sum_{j=1}^dy_jdx_j))$ is the prequantization of the symplectic toric manifold $(\R^{2d},\sum_{j=1}^ddx_j\wedge dy_j)$ 
(with the standard $\T^d$-action on $\R^{2d}\cong \C^d$).
Every Lagrangian toric fiber in $\R^{2d}$ is displaceable by Hamiltonian isotopies: for example, by an appropriate translation in one of the coordinates. 
Thus, by Proposition \ref{lift.preq}, every pre-Lagrangian toric fiber in $\S^1\times\R^{2d}$ is also displaceable.
\end{example}
\subsection{The method of probes.}\label{method of probes}
The method of probes was introduced by McDuff \cite{mcduff} and serves to displace some Lagrangian toric fibers in symplectic toric manifolds. We briefly recall it here, since the results of Section \ref{section displace by cuts} will be proved by combining Proposition \ref{lift.preq} with displaceability results in symplectic toric manifolds obtained via this method.

Let  $\Delta =\mu(M) \subset \R^n$ be the Delzant polytope corresponding to some symplectic toric manifold $M^{2n}$ with moment map $\mu$.
Take any facet $F$ of $\Delta \subset \R^n$ and denote its inward normal by $\eta_F$. An integral vector $\lambda \in \Z^n$ is called integrally transverse to $F$ if $|\langle \lambda, \eta_F \rangle |=1$. The {\bf probe} $p_{F, \lambda}(w)$ with direction $\lambda \in \Z^n$  and initial point $w \in F$ is the half open line segment consisting of $w$ and all the points in $\ii \Delta$ that lie on the ray from $w$ in the direction $\lambda$. McDuff in Lemma 2.4 of \cite[Lemma 2.4]{mcduff} proved that if $w \in \ii F$ and $u \in \ii \Delta$ is any point which lies on  the probe $p_{F, \lambda}(w)$, less than halfway along it, then the Lagrangian toric fiber in $M$ corresponding to $u$ is displaceable. Moreover, this fiber can be displaced by an isotopy of $M$ supported in a compact subset of $\mu^{-1}(p_{F, \lambda}(w))$.

Using this method McDuff \cite{mcduff} and Abreu-Borman-McDuff \cite{abreubormanmcduff} showed that any compact symplectic toric manifold contains uncountably many displaceable Lagrangian toric fibers. Combining this with Proposition \ref{lift.preq} we see that any contact toric manifold $(V,\xi)$ which is the prequantization of a compact symplectic toric manifold contains uncountably many displaceable pre-Lagrangian toric fibers.

In particular this implies that a contact toric manifold for which all pre-Lagrangian toric fibers are non-displaceable (for instance, $\T^k\times\mathbb{S}^{2d+k-1}$ as we will see in \ref{section displace by cuts}) cannot be the prequantization of a compact symplectic toric manifold.


 \begin{remark}\label{stem}
 According to the definition given by Entov and Polterovich in \cite{entpol}, a Lagrangian toric fiber in a symplectic toric manifold is a {\bf stem} if any other Lagrangian toric fiber is displaceable. 
 Translating to the contact setting one obtains the following definition:
a pre-Lagrangian toric fiber is a (contact) stem if every other pre-Lagrangian toric fiber is displaceable. From Proposition \ref{lift.preq} we see that if $L\subset M$ is a stem, then $\pi^{-1}(L)\subset V$ is a stem. While
 in the symplectic setting a stem, if it exists, is unique and non-displaceable  (\cite{entpol}), in the contact setting this is not necessarily true. In $\mathbb{S}^{2d-1}$ and $\T^k\times \S^{2d+k-1},$ $d\geq2$, every pre-Lagrangian toric fiber is a stem and none of them is non-displaceable (see Section \ref{displace not ord}). However, as was proved by Borman and Zapolsky, in the family of contact toric manifolds admitting a monotone quasimorphism with the vanishing property stems behave as in the symplectic case: if a stem exists then it is unique and non-displaceable (see Corollary 1.19 and Corollary 1.11 in \cite{BormanZapolsky}). The prequantization of any even monotone symplectic toric manifold (with the symplectic form scaled appropriately) admits such a quasimorphism (see Theorem 1.3. in \cite{BormanZapolsky}). \eor
 \end{remark}
 \begin{example}\label{example displace via prequantization} The central fiber in $\mathbb{CP}^{d-1}$, the Clifford torus
$$\T_{\mathbb{CP}}=\{[z_1,\ldots,z_{d}]\in\mathbb{CP}^{d-1}|\hskip1mm |z_1|^2=\cdots=|z_{d}|^2\}$$ is proved to be non-displaceable (Cho-Poddar  \cite{ChoPoddar}). On the other hand, using the method of probes, it can be shown that all other Lagrangian toric fibers in $\mathbb{CP}^{d-1}$ are displaceable. Therefore the Clifford torus is a stem. By Remark \ref{stem} the pre-images of the Clifford torus in $\mathbb{S}^{2d-1}$, $\mathbb{RP}^{2d-1}$ and $L_p^{2d-1}=\mathbb{S}^{2d-1}/ \bb{Z}_p$  (under the prequantization map) are also stems.  However their displaceability properties are very different. The real projective space admits a monotone quasimorphism with the vanishing property and therefore its stem is non-displaceable \cite{BormanZapolsky}. It is expected that all lens spaces also admit such quasimorphism (\cite{GKPS} in preparation). That would imply that the preimage of the Clifford torus in any $L_p^{2d-1}$ is also non-displaceable.
In the case of $\mathbb{S}^{2d-1},$ the preimage of the Clifford torus is displaceable as shown in Proposition \ref{displace sphere}. Absence of a non-displaceable pre-Lagrangian toric fiber also implies that $\mathbb{S}^{2d-1}$ does not admit a monotone quasimorphism with the vanishing property (see Theorem 1.14. in \cite{BormanZapolsky}).
\end{example}
\subsection{Contact reduction of a contact toric manifold and displaceability}\label{section contact reduction}
In this Section we explain how one can use contact reduction to deduce certain results about (non)-displaceability.
This is done for the sake of completeness, as in this article we are not applying this method.

We first recall the notion of a
\textbf{contact reduction} (for more details see \cite{Albert} or Section 7.7. in \cite{Geiges}). 
Suppose that a compact Lie group $G$ acts on a contact manifold $(\widetilde{V},\widetilde{\xi}=\ker\widetilde{\alpha})$ preserving the contact form $\widetilde{\alpha}$ and let $\mu_G$ be the corresponding moment map.
Assume moreover that $0\in\mathfrak{g}^{\ast}$ is a regular value and that $G$ acts freely and properly on the level ${\mu_G}^{-1}(0).$ 
Let $\rho:{\mu_G}^{-1}(0)\rightarrow{\mu_G}^{-1}(0)/G$ be the quotient map. 
The contact form $\widetilde{\alpha}$ naturally induces a contact form $\alpha$ on $V$ such that $\rho^*\alpha=\widetilde{\alpha}$ on ${\mu_G}^{-1}(0).$
Moreover, if $ \widetilde{V}^{2d+1}$ is a contact toric manifold with a toric $\T^{d+1}$-action, 
and $G=\S^1$ is a subgroup of $\T^{d+1}$,
then the residual torus $\T^{d+1}/G$ acts on the reduced space $V$, turning it into a contact toric manifold.

The results of Abreu and Macarini (\cite{abreumacarini}),
establishing a connection between rigidity of Lagrangian toric fibers in a symplectic manifold and in its reduction,
can easily be translated to the contact setting. Proposition \ref{displace reduction} follows immediately from the following lemma, whose proof is omitted since it is analogous to the one in the symplectic case  (see Section 3 in \cite{abreumacarini}).
 \begin{lemma}\label{lift reduction}(\emph{Lifting property for contact reduction}) Let $\varphi\in \mathrm{Cont}_0(V,\xi)$. Then, there is $\widetilde{\varphi}\in\mathrm{Cont}_0(\widetilde{V},\widetilde{\xi})$ such that
 $\rho\circ\widetilde{\varphi}=\varphi\circ\rho$.
 \end{lemma}
\begin{proposition}\label{displace reduction} If $L\subset V$ is a displaceable pre-Lagrangian toric fiber, then $\widetilde{L}=\rho^{-1}(L)\subset\widetilde{V}$ is a displaceable pre-Lagrangian toric fiber.
\end{proposition}
\subsection{Contact cuts and displaceability.}\label{section contact cuts}
The procedure of contact cutting was defined by Lerman in \cite{Lcuts} who also proved the following result.
\begin{theorem}[Theorem 2.11 in \cite{Lcuts}]
Let $(V, \alpha)$ be a contact manifold with an action of $\S^1$ preserving $\alpha$ and let $\mu_{\alpha}$ denote the corresponding 
$\alpha$-moment map. Suppose that $\S^1$ acts freely on the zero level set $\mu_{\alpha}^{-1}(0)$. Then the {\bf cut manifold}, defined as
$$V_{[0,\infty)}=\mu_{\alpha}^{-1}([0,\infty))/\sim,$$
where $x_1\sim x_2$ if and only if $\mu_{\alpha}(x_1)=\mu_{\alpha}(x_2)=0$ and the points $x_1$ and $ x_2$ are in the same circle orbit,
is naturally a contact manifold. Moreover, the natural embedding of the reduced space $V_0=\mu_{\alpha}^{-1}(0)/\S^1$ into 
$V_{[0,\infty)}$ is contact and the complement $V_{[0,\infty)}\setminus V_0$ is contactomorphic to the open subset 
$\{x \in V;\ \mu_{\alpha}(x)>0\}$ of $(V, \alpha)$.
\end{theorem}
If the above $\S^1$ is a subcircle of a torus $\T^{\frac{1}{2}(\dim V+1)}$ acting on $V$ in a Hamiltonian way (thus giving $V$ a structure of a contact toric manifolds) then the $\T^{\frac{1}{2}(\dim V+1)}$ action restricts to an action on the cut manifold $V_{[0,\infty)}$ turning it also into a contact toric manifold. The moment cone corresponding to $V_{[0,\infty)}$ is the intersection of the cone of $V$ in $\R^{\frac{1}{2}(\dim V+1)}=Lie(\T^{\frac{1}{2}(\dim V+1)})^*$ with the half space 
$\{x \in \R^{\frac{1}{2}(\dim V+1)};\ \langle x, \xi \rangle \geq 0\}$ where $\xi \in Lie(\T^{\frac{1}{2}(\dim V+1)})$ is the infinitesimal generator of the chosen $\S^1$.
The following easy observation will be used repetitively throughout Section \ref{section displace by cuts}.
\begin{lemma}\label{extend from cut}
Any contactomorphism of $V_{[0,\infty)}$, compactly supported in $V_{[0,\infty)}\setminus V_0$, can be extended to a contactomorphism of the whole $V$.
Therefore, if $L \subset \mu_{\alpha}^{-1}((0,\infty)) \subset V $ is a pre-Lagrangian displaceable in $V_{[0,\infty)}$ by a contact isotopy supported in $V_{[0,\infty)}\setminus V_0$, then it is displaceable in $V$.
\end{lemma}
\begin{example}\label{example cut}
The standard $\T^2$ action on the contact toric manifold $V:=\S^1 \times \S^2 \subset \S^1 \times \R \times \C$ comes from the standard $\S^1$ actions on $\S^1$ and on $\C$. The associated moment cone is presented on the left picture in Figure \ref{figure cut}.
(This space is a representative of the family $\T^k \times \S^{2d+k-1}$ discussed in Section  \ref{section displace by cuts}.)
Performing a contact cut with respect to the circle $\S^1=\{(t,t) \in \T^2\}$ we obtain a contact toric manifold $V_{[0,\infty)}$
whose moment cone is presented on the right picture.
\begin{figure}[h]
\includegraphics[width=1\textwidth]{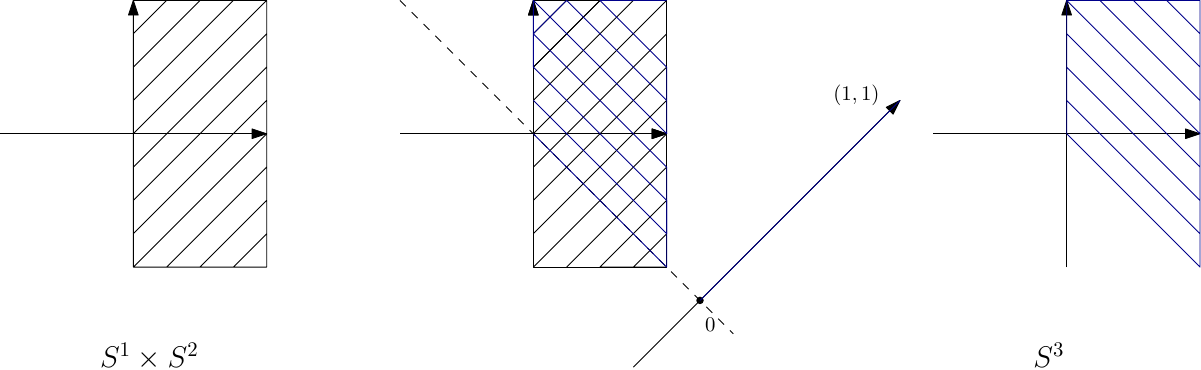}
\caption{The contact cut of $\S^1 \times \S^2$ with respect to $\S^1=\{(t,t) \in \T^2\}$ and the resulting in $\S^3$.}
\label{figure cut}
\end{figure}
This cone differs from the moment cone of $\S^3$ with the standard contact structure, viewed as the prequantization of $(\C\P^1, \frac{1}{\pi} \omega_{FS})$, by a $GL(2,\Z)$ transformation (see Example \ref{example sphere is a prequantization}). 
Therefore $V_{[0,\infty)}$ is contactomorphic to $(\S^3, \xi_{st})$ and the toric action differs (from the 
action on $\S^3$ viewed as the prequantization of $\C\P^1$) just by a reparametrization of the torus. In particular the torus orbits remain unchanged.
Using Proposition \ref{lift.preq} we can 
to lift isotopies of $\C\P^1$ obtained by the method of probes
(Section \ref{method of probes}) to isotopies of $(\S^3, \xi_{st})$.
These isotopies displace pre-Lagrangians toric fibers corresponding to the rays contained in the green region in Figure \ref{figure displaceable}.
Moreover, as they are supported in  $ V_{[0,\infty)}\setminus V_0 $, they can be extended to isotopies of $\S^1 \times \S^2$. Therefore all pre-Lagrangian toric fibers in $\S^1\times \S^2$ which map to the green region are displaceable.
\begin{figure}[h]
\centering
\includegraphics[width=.25\textwidth]{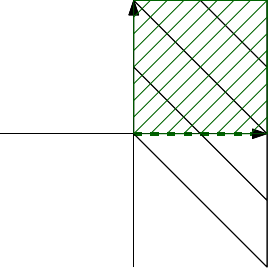}
\caption{Pre-Lagrangians displaceable in $\S^3$ by isotopies coming from the probes methods.}\label{figure displaceable}
\end{figure}
\end{example}
\section{Displaceability of fibers in non-orderable manifolds}\label{displace not ord}

In this Section we prove that for all contact toric manifolds that are presently known to be not orderable all pre-Lagrangian toric fibers are displaceable. This observation suggest that the lack of orderability of a contact toric manifold
implies displaceability of its toric fibers. Our proof is done on a case by case basis. Therefore it, unfortunately, does not explicitly show how the lack of orderability could affect the displaceability of the fibers.

To the best of our knowledge the only contact toric manifolds which are {\it known} to be non-orderable are:
\begin{itemize}
\item Contact spheres $(\mathbb{S}^{2d-1},\xi_{st} )$, with $d\geq2$, equipped with the standard contact structure $\xi_{st}=T \mathbb{S}^{2d-1}\cap J(T\mathbb{S}^{2d-1})$.
\item  Contact toric manifolds $\T^k \times \S^{2d+k-1}$, with $d\geq 2$, $k> 0$, with contact structure described explicitly in Section \ref{section displace by cuts}.
\end{itemize}
The first ones are ideal contact boundaries of $\C^d$, i.e. the $d$-stabilization of a point.
The second ones are ideal contact boundaries of $(T^* \S^1)^k \times \C^d$ i.e. the $d$-stabilization of the Liouville manifolds $(T^* \S^1)^k$.
They are not orderable by a result of Eliashberg-Kim-Polterovich (Theorem 1.16 in \cite{EKimP}). 
It remains an open question whether $\T^k \times \S^{k+1}$ are orderable, even in the case $k=1$.
\subsection{Contact sphere}\label{contact sphere}
 The standard contact sphere $(\mathbb{S}^{2d-1},\xi_{st} )$ can be presented in complex coordinates as
$\mathbb{S}^{2d-1}=\{(z_1,\ldots,z_d)\in\mathbb{C}^d\hskip1mm|\hskip1mm \sum_{j=1}^{d}|z_j|^2=1\}$ with $\xi_{st}$ given as the kernel of $\alpha_{st}:=\frac{i}{4}\sum_{j=1}^d(z_jd\bar{z}_j-\overline{z}_jdz_j).$
The standard $\T^d$-action on $\mathbb{S}^{2d-1}$ defined by
 $$(t_1,\ldots,t_d)\ast (z_1,\ldots ,z_d)\longmapsto(t_1 z_1,\ldots, t_d z_d)$$
 makes the sphere $(\mathbb{S}^{2d-1},\xi_{st})$ a contact toric manifold. The moment map with respect to $\alpha_{st}$ is
 $\mu_{\alpha_{st}}(z_1,\ldots,z_d)=\pi(|z_1|^2,\ldots,|z_d|^2).$ Thus, the moment cone is $\R_{\geq0}^{d}.$
 A pre-Lagrangian toric fiber in $(\mathbb{S}^{2d-1},\xi_{st} )$ is any submanifold
 \begin{equation}\label{fiber in the sphere}
 L_{c_1,\ldots,c_d}=\{(z_1,\ldots,z_d)\in\mathbb{S}^{2d-1}\hskip1mm|\hskip1mm |z_j|=c_j,\hskip1mm j=1,\ldots,d\}
 \end{equation}
 where $c_j\in(0,1)$ are constants such that $\sum_{j=1}^{d}c_j^2=1.$ We will now prove
 Proposition \ref{displace sphere}, i.e. the fact that all pre-Lagrangian toric fibers are displaceable.
 In fact a stronger result is true.
We are grateful to Patrick Massot and the referee who independently pointed out to us that, as long as one does not require the displacing isotopy to displace {\it all} pre-Lagrangian toric fibers simultaneously, one can construct an isotopy displacing one particular fiber in the following way.
(Our original idea, providing one isotopy simultaneously displacing all the fibers, is explained in Remark \ref{displace all with one iso}.) 
\begin{proposition}\label{displace all closed}
 Every closed proper subset of $(\S^{2d-1},\xi_{st})$, with $d\geq 2$, is displaceable. 
 \end{proposition}
 \begin{proof}
  Let $L\subset \S^{2d-1}$ be a closed proper subset and let $p\in (\S^{2d-1} \setminus L)$ be any point. 
 The key point in this proof is the observation that $(\S^{2d-1}\setminus \{p\},\xi_{st})$ is contactomorphic to $\R^{2d-1}$ with its standard contact structure $\xi_0=\ker(dx_{d}+\sum_{j=1}^{d-1}x_j\cdot dy_j)$ 
 (by Proposition 2.1.8 and Example 2.1.3 in \cite{Geiges}). 
 Denote this contactomorphism by 
 $$\psi \colon (\S^{2d-1}\setminus\{p\},\xi_{st}) \rightarrow (\R^{2d-1},\xi_0).$$
Observe that any translation in the $x_d$ direction is a contactomorphism of $(\mathbb{R}^{2d-1}, \xi_0).$
As $L$ is compact, there exists $R>0$ such that
the contact isotopy $\{\phi_t\}_{t\in[0,1]}$, with $\phi_t$ being the translation in $x_d$ direction by $tR$, displaces $\psi(L)$. Let $h_t$ denote the contact Hamiltonian generating the isotopy $\phi_t$.
Let $\delta$ be a cut-off function on $\R^{2d-1}$ which is $1$ on a neighborhood of $\cup_{t\in[0,1]}\phi_t(\psi(L))$, and $0$ outside of a compact set. Then the isotopy generated by the contact Hamiltonian $\delta h_t$ displaces $\psi(L)$. Moreover, as it is compactly supported in $\R^{2d-1}=\psi(\S^{2d-1}\backslash\{p\})$, the isotopy $\psi^{-1} \circ \phi_t \circ \psi$ of $\S^{2d-1}\backslash\{p\}$ can be extended to an isotopy of the whole sphere $\S^{2d-1}$, displacing $L$.
 \end{proof}
 
\begin{remark}\label{displace all with one iso}
In fact all pre-Lagrangian toric fibers can be displaced using only one contact isotopy. 
 In the proof of their non-orderability result, Eliashberg, Kim and Polterovich build a positive contractible loop as a composition of certain contactomorphisms. One of these building blocks is an isotopy called a distinguished contact isotopy.  In the case of $\S^{2d-1}$ they give an explicit formula for this isotopy. Moreover Giroux in \cite{Giroux} gives another formula for a contact isotopy of $\S^{2d-1}$ playing the same role in the construction of a positive contractible loop. This contact isotopy displaces all pre-Lagrangian toric fibers in $\S^{2d-1}$, $d\geq 2$, as we now explain.
  Consider the map $\tau_t:B^{2d}\rightarrow \bb{C}^{d}, t\geq0$ defined by
 $$\tau_t(z_1,\ldots,z_d)=\frac{1}{\cosh t+z_1\sinh t}(\sinh t+z_1\cosh t,z_2,\ldots,z_d),$$
 where $B^{2d}$ is the unit ball and 
 $\cosh t=\frac{e^t+e^{-t}}{2},$ $\sinh t=\frac{e^t-e^{-t}}{2}$  (\cite{Giroux}).
 We first observe that the map $\tau_t$ is well defined because $\cosh t+z_1\sinh t \neq0$ for every $|z_1|\leq1$ and every $t\geq0.$ Indeed, if $t=0$ the expression is equal to 1. Suppose there is some $t>0$ and $z_1=x+iy$ such that $\cosh t+z_1\sinh t =0.$ That would imply that $\cosh t+ x\sinh t=0$ and $y\sinh t=0.$  Consider the function $f(x)=\cosh t+x\sinh t$ on the domain $|x|\leq1,$ for a fixed constant $t>0.$ We have $f(-1)=e^{-t}>0$ and $f'(x)=\cosh t>0,$ hence $f(x)>0$ for all $|x|\leq1.$
By a straightforward calculation we check that $\tau_t$ is a complex automorphism of the unit ball $(B^{2d},\frac{i}{2}\sum dz_j\wedge d\bar{z}_j)$.
Since every complex automorphism of the unit ball $B^{2d}\subset \mathbb{C}^d$ restricts to a contactomorphism of its boundary, i.e. the standard contact sphere $(\mathbb{S}^{2d-1},\xi_{st})$, it follows that
  $\tau_t$ is a contactomorphism on $\mathbb{S}^{2d-1}$ for every $t\geq0.$ Since $\tau_0=id$, the map $\tau_t$ is a well defined contact isotopy starting at the identity.
 Take any pre-Lagrangian toric fiber $L=L_{c_1,\ldots,c_d}$ given by (\ref{fiber in the sphere}). We claim that for $t>0$ large enough
 $$L\cap\tau_t(L)=\emptyset.$$
 It is enough to show that there exists $t>0$ big enough so that for any $z_1$ with $|z_1|=c_1 \in (0,1)$, we have $|\cosh t+z_1\sinh t|\neq1$, as this implies that the norm of the second coordinate changes after applying $\tau_t$, and thus such $\tau_t$ displaces $L$. 
 Write $z_1=x+iy,$ hence $x^2+y^2=c_1^2$, and
 consider the function $f_t(x)=|\cosh t+z_1\sinh t|^2-1$, for $|x|\leq c_1.$
 Note that
 \begin{align*}
 f_t(x)&=|\cosh t+z_1\sinh t|^2-1=(\cosh t+x\sinh t)^2+(y\sinh t )^2-1\\
 &=\cosh^2t+2x\cosh t\sinh t +c_1^2\sinh^2t-1=\sinh^2 t+2x\cosh t\sinh t +c_1^2\sinh^2t
 \\
 &=\sinh t(2x\cosh t+(1+c_1)\sinh t)
 \end{align*}
 Since $\sinh t>0$ for $t>0$, it is enough to show that for $t$ big enough the expression $(2x\cosh t+(1+c_1)\sinh t)$
  is positive for all $|x|\leq c_1<1$. 
 Take $t>0$ such that $\tanh t>\frac{-2x}{1+c_1}.$ Such $t$ always exists since $\tanh t$ tends to infinity when $t$ is large.
 For this choice of $t$ we have that $f_t(x)>0$ for every $|x|\leq c_1<1$, proving that $\tau_t$ displaces $L$.\eor
\end{remark}
\begin{remark} Note that we showed $L\cap\tau_t(L)=\emptyset$ by comparing the norm of the second coordinate, thus we used the fact that $d\geq 2$. In the case $d=1$ the map $\tau$ is a well defined contactomorphism of a circle. However as the only pre-Lagrangian of $\S^1$ is the whole $\S^1$, it is trivially non-displaceable.
\eor
\end{remark}
\subsection{Displaceability of fibers in $\T^k \times \S^{2d+k-1}$.}\label{section displace by cuts}
Consider the contact toric manifold $\T^k \times \S^{2d+k-1}$\label{tks2dk} described as
$$\{ (e^{2\pi i\theta_1},\ldots,e^{2\pi i\theta_{k}}, x_1,\ldots, x_k, z_1,\ldots, z_d) \in \T^k \times \R^k \times \C^d;\, \sum_{l=1}^k x_l^2 + \sum_{j=1}^d |z_j|^2=1\}$$
with the contact structure given by the kernel of the form
$$\beta_k=\sum_{l=1}^k x_l d\theta_l+\frac{i}{4}\sum_{j=1}^d(z_jd\bar{z}_j-\overline{z}_jdz_j).$$
This is the ideal contact boundary of $T^*\T^k \times \C^d$, and thus is non-orderable 
if $d\geq 2$ by Theorem 1.16 in \cite{EKimP}.
The toric action of $\T^{d+k}$ is given by
$$(t_1,\ldots,t_d,s_1,\ldots,s_k)\ast  (e^{2\pi i\theta_1},\ldots,e^{2\pi i\theta_{k}}, x_1,\ldots, x_k, z_1,\ldots, z_d)=$$
$$(s_1 e^{2\pi i\theta_1},\ldots,s_ke^{2\pi i\theta_{k}}, x_1,\ldots, x_k, t_1z_1,\ldots, t_dz_d).$$
 The corresponding  $\beta_k$-moment map, 
 $\mu_{\beta_k}\colon \T^k \times \S^{2d+k-1} \rightarrow Lie(\T^{d+k})^*= \R^{d+k}$ is
$$\mu_{\beta_k} (e^{2\pi i\theta_1},\ldots,e^{2\pi i\theta_{k}}, x_1,\ldots, x_k, z_1,\ldots, z_d)=\pi(|z_1|^2,\ldots,|z_d|^2,x_1,\ldots,x_k),$$
thus the moment cone is $(\R_{\geq 0})^d \times \R^k.$
Pre-Lagrangian toric fibers in $\T^k\times\mathbb{S}^{2d+k-1}$ 
are submanifolds $L=L_{c_1,\ldots,c_d}$ given by
 $$\{ (e^{2\pi i\theta_1},\ldots,e^{2\pi i\theta_{k}}, x_1,\ldots, x_k, z_1,\ldots, z_d)\in \T^k\times\mathbb{S}^{2d+k-1}\hskip1mm;\hskip1mm |z_j|=c_j,\hskip1mm j=1,\ldots,d\}$$
for some $x_1,\ldots, x_k\in\R$ and $c_1,\ldots,c_d>0$ such that $\sum_{l=1}^kx_l^2+\sum_{j=1}^dc_j^2=1$, and thus they correspond to the rays $\{t(c_1,\ldots,c_d,x_1,\ldots,x_k);\ t \in \R_+\}$
in 
$$(\R_{> 0})^d \times \R^k \subset Cone(\mu_{\beta_k}(\T^k \times \S^{2d+k-1})).$$

In this subsection we show that all pre-Lagrangian toric fibers in $(\T^k \times \S^{2d+k-1},\ker\beta_k)$, $d\geq 1$, $k\geq 1$, are displaceable.
To construct a contact isotopy displacing a given pre-Lagrangian toric fiber $L$ of $\T^k \times \S^{2d+k-1}$ we will look at a manifold obtained via contact cutting $\T^k \times \S^{2d+k-1}$ with respect to appropriate circles. 
\begin{proof}[Proof of Theorem \ref{displace in boundary of stabilization} with $d \geq 2$.]  
 \hfill \break 
 {\bf Case $k=1$.}
We start by analyzing $\S^1 \times \S^{2d}$.
The cone corresponding to this contact toric manifold is 
$C=C_{\S^1 \times \S^{2d}}= \R_{\geq 0}^d \times \R,$ and its outward normals are $-e_1,\ldots, -e_d$.
Perform a contact cut with respect to the diagonal circle in $\T^{d+1}$ (see Section \ref{section contact cuts}).
Denote the moment map for that circle by $\mu^+$, i.e. $\mu^+(\theta, x, z_1,\ldots, z_d) \mapsto \pi(x +\sum_{j=1}^d |z_j|^2)$,
and the resulting cut manifold
by 
$$M^+:=(\mu^+)^{-1}([0,\infty))/\sim.$$ 
It is a contact toric manifold corresponding to the convex cone, $C_1=C_1^+$, 
obtained from $C$ by cutting it with the hyperplane perpendicular to the vector $-\sum_{l=1}^{d+1} e_l$.
This means that $C_1$ is spanned by the directions:
$e_{d+1}$ and $e_i-e_{d+1}$ for $i=1,\ldots,d$.
In fact it is the contact toric sphere $\S^{2d+1}$. Indeed,
as we have already seen in Example \ref{example sphere is a prequantization} contact sphere $\S^{2d+1}$, viewed as the prequantization of $(\mathbb{CP}^{d}, \frac{1}{\pi}\omega_{FS})$, is the contact toric manifold corresponding to the cone, which we call $C_2$, whose edges are in the directions $e_i+e_{d+1},$ for $i=1,\ldots,d$.
The $GL(d+1,\Z)$ transformation given by the following lower triangular matrix (only the last row and the diagonal have non-zero entries) maps the cone $C_1$ to the cone $C_2$.
\begin{displaymath}
A_1=\left[
\begin{array}{rrrrr}
\,1\, && &&\\
&\,1\, & &&\\
&&\ldots &&\\
&&&\,1\, &\\
\,2\,&\,2\,&\ldots&\,2\,&\,1\\
\end{array} \right].
\end{displaymath}
Note that $d=1$ is exactly the case described in Example \ref{example cut} and on Figure \ref{figure cut}.

To construct isotopies displacing pre-Lagrangian toric fibers in $\S^{2d+1}$ we will use Lemma \ref{lift} and isotopies displacing Lagrangian toric fibers in $(\mathbb{CP}^{d}, \frac{1}{\pi}\omega_{FS})$ obtained via McDuff's method of probes (see Example \ref{method of probes}).
The point of using these particular isotopies of $\S^{2d+1}$ (instead of, for example, isotopies from Proposition \ref{displace all closed}) is that we know their support: it is far from the ``cut" and therefore these isotopies of $\S^{2d+1}$ are extendable to isotopies of $\S^1 \times \S^{2d}$.

Recall that $(\mathbb{CP}^{d}, \frac{1}{\pi}\omega_{FS})$ is a symplectic toric manifold whose moment map image is the $d$-dimensional simplex of size $1$, $\Delta^d(1)$. For each $i=1,\ldots,d$ let $F_i$ denote the facet of $\Delta^d(1)$ whose outward normal is $-e_i$. The vector $e_i$ can be used as a direction of a probe. This way we can displace all Lagrangian toric fibers in $\mathbb{CP}^{d}$ corresponding to points $x\in \Delta^d(1) \subset \R^d$ such that
$$x= \frac{1}{2}a_ie_i+\sum_{l\neq i}a_le_l,\textrm{ with }\ \ a_1,\ldots,a_d>0$$
(i.e. to the points in the interior of the convex hull of points  $0, \frac 1 2 e_i,$ and $ \{e_l\}_{l\neq i}$)
by an isotopy with support contained in the preimage of $\ii \Delta^d(1) \cup \ii F_i$. 
Lifting this isotopy to the sphere we can displace any pre-Lagrangian toric fiber corresponding to 
a ray through $y \in C_2\subset \R^{d+1}$ such that
$$y= a_i(2e_{d+1}+e_i)+a_{d+1}e_{d+1}+\sum_{l\neq i,\,d+1}a_l(e_l+e_{d+1}),\textrm{ with }\ \ a_1,\ldots,a_{d+1}>0$$
(i.e. of the interior of the cone spanned by the directions
$$e_1+e_{d+1},\ldots, e_{i-1}+e_{d+1},e_i+2e_{d+1},e_{i+1}+e_{d+1},\ldots, e_d+e_{d+1},e_{d+1})$$
with an isotopy supported on the preimage of $\ii C_2 \cup \ii \widetilde{F}_i$, where $ \widetilde{F}_i$ is the facet of $C_2$ with normal $-e_i$. (Figure \ref{figure displaceable} in Example \ref{example cut} presents this set for the case $d=1$.)
The map $A_1^{-1}$, mapping the cone $C_2$ to the cone $C_1$ of the cut space $M_+$,  maps this set to the set of points 
$$y= a_ie_i+a_{d+1}e_{d+1}+\sum_{l\neq i,\,d+1}a_l(e_l-e_{d+1}),\textrm{ with }\ \ a_1,\ldots,a_{d+1}>0.$$
 These are the points $y=(y_1,\ldots,y_{d+1}) \in C_1$ with $y_l>0$ for $l \in \{1, \ldots, d\}$, and 
 $y_{d+1}> -\sum_{l \neq i,\ d+1} y_l  .$
 In particular this includes 
 $$(\R_{> 0})^d \times \R_{\geq 0}, \textrm{ if } d>1, \textrm{ and } (\R_{> 0})^1\times \R_{>0}, \textrm{ if } d=1.$$
 Note that all these isotopies are supported away from the facet of $C_1$ with normal $-\sum_{l=1}^{d+1}e_l$.
 Therefore, by Lemma \ref{extend from cut}, these isotopies can be extended to isotopies of the whole $\S^1 \times \S^{2d}$.
 
 We now repeat the whole construction but starting with contact cutting with respect to the circle $(t,\ldots,t, t^{-1}) \in \T^d \times \S^1$.
 The moment map for that circle is 
$$\mu^-(\theta, x, z_1,\ldots, z_d) \mapsto \pi(-x +\sum_{i=1}^d |z_i|^2).$$
Thus the cut manifold $M^-:=(\mu^-)^{-1}([0,\infty))/\sim$ is a contact toric manifold corresponding to the cone, $C^-_1$, spanned by the edges $-e_{d+1}$ and $e_i+e_{d+1}$ for $i=1,\ldots,d$.
The resulting cut space is again a toric contact sphere $\S^{2d+1}$. 
 Repeating the whole process we construct isotopies of $\S^1 \times \S^{2d}$ that displace, in particular, pre-Lagrangian toric fibers corresponding to the rays in 
 $(\R_{>0})^{d}\times \R_{\leq 0}$ if $d\geq 2$, and to the rays in $\R_{>0}\times \R_{< 0}$ if $d=1$.
 
 Putting these two steps together we are able to displace {\it all} pre-Lagrangian toric fibers of $\S^1 \times \S^{2d}$ if $d \geq 2$,
 and all fibers in $\S^1 \times \S^{2}$ {\it apart} from the fiber corresponding to the ray $\{(y_1,0);\ y_1 \in \R_{\geq 0}\}$ if $d=1$.
 It will be shown later that this fiber is also displaceable.
 
 {\bf General case.} 
 To  obtain a sphere as a cut manifold of $\T^k \times \S^{2d+k-1}$ and displace certain fibers as we did above for the case $k=1$, we need to cut $\T^k \times \S^{2d+k-1}$ $k$ times in $2^k$ different ways. To keep a record of the cuts we introduce the following notation. For any $j \in \{1,\ldots,k\}$ and any $\varepsilon_j \in \{-1,1\}$, let $\S^1_{j,{\varepsilon_j}}$ denote the circle subgroup of $\T^d \times \T^k$ generated by 
$$(t,\ldots,t; 1,\ldots,1,t^{\varepsilon_j},1,\ldots, 1),$$
where $t^{\varepsilon_j}$ is at the position $d+j$.
For any $\varepsilon=(\varepsilon_1,\ldots,\varepsilon_k) \in \{-1,1\}^k$ let $M^{\varepsilon}$ denote the contact toric manifold obtained from $T^k \times \S^{2d+k-1}$ by consecutive performing $k$ contact cuts, with respect to circles $\S^1_{1,{\varepsilon_1}}$, ..., $\S^1_{k,{\varepsilon_k}}$. We denote the corresponding moment cone by $C^{\varepsilon}_1$.
Observe that $C^{\varepsilon}_1$ is the cone in $\R^{d+k}$ with $(d+k)$ facets whose outward normals are
$$-e_1,\ldots, -e_d, \evj:=-\sum_{i=1}^d e_i - \varepsilon_j e_{d+j}, \ j=1,\ldots, k.$$
The edges of this cone have directions
$$ \varepsilon_j e_{d+j},\ j=1,\ldots, k,$$
and
$$e_i-\sum_{j=1}^k \varepsilon_j e_{d+j},\ i=1,\ldots,d.$$
The facets of $C^{\varepsilon}_1$ created by cutting are the ones with normals $\evj, \ j=1,\ldots, k.$
Thus there is a contactomorphism between the preimage in $M^{\varepsilon}$ of $C^{\varepsilon}_1\setminus \{\textrm{ facets with normals }\evj, \ j=1,\ldots, k\}$, which we denote by $U^{\varepsilon}$, and an open subset of $\T^k \times \S^{2d+k-1}$.
Any contact isotopy of $M^{\varepsilon}$ compactly supported in $U^{\varepsilon}$ can be extended to a contact isotopy of the whole $\T^k \times \S^{2d+k-1}$.

Note that the space $M^{\varepsilon}$ is contactomorphic to the standard contact sphere $\S^{2d+2k-1}$ and the action differs only by a reparametrization of the torus.
Indeed, the $GL(d+k,\Z)$ transformation given by the following matrix
\begin{displaymath}
A^{\varepsilon}_1=\left[
\begin{array}{ccc|cccc}
1 && &&&&\\
&\ddots& &&&&\\
&&1&&&&\\
\hline
1&\ldots&1&\varepsilon_1&&&\\
&\ldots&&&\ddots&\\
1&\ldots&1&&&\varepsilon_{k-1}&\\
1+k&\ldots&1+k&\varepsilon_1&\ldots&\varepsilon_{k-1}&\varepsilon_{k}
\end{array} \right]
\end{displaymath}
maps  
$$e_i-\sum_{j=1}^k \varepsilon_j e_{d+j} \mapsto  e_i+e_{d+k},\ \ i=1,\ldots,d$$
$$\varepsilon_j e_{d+j} \mapsto e_{d+j}+e_{d+k},\ \ j=1,\ldots, k-1,$$ 
$$\varepsilon_k e_{d+k} \mapsto  e_{d+k}.$$
Therefore $A^{\varepsilon}_1$ maps the cone $C^{\varepsilon}_1$ to the cone, that we denote by $C^{\varepsilon}_2$, of  the contact toric sphere $\S^{2d+2k-1}$ viewed as the prequantization of $(\C\P^{d+k-1}, \frac{1}{\pi} \omega_{FS}) $.

Observe that the facet of $C^{\varepsilon}_1$ with normal $\evj,$ $j=1,\ldots, k-1$, is the cone with the edges in the directions $\varepsilon_l e_{d+l}$, $l\in \{1,\ldots,k\}\setminus \{j\}$ and $e_i-\sum_{j=1}^k \varepsilon_j e_{d+j},\ i=1,\ldots,d.$ Such facet is mapped by $A^{\varepsilon}_1$ to the facet of $C^{\varepsilon}_2$ which is the cone with the edges in the directions $e_{d+l}+e_{d+k}$, $l\in \{1,\ldots,k\}\setminus \{j,k\}$, $e_{d+k}$ and $e_i+e_{d+k}$ for $i=1,\ldots,d$, i.e.
the facet of $C^{\varepsilon}_2$ with normal $-e_{d+j}$. 
The facet of $C^{\varepsilon}_1$ with normal $\eta_{\varepsilon_k},$ i.e. the cone whose edges are in the directions $\varepsilon_l e_{d+l}$, $l\in \{1,\ldots,k-1\}$ and $e_i-\sum_{j=1}^k \varepsilon_j e_{d+j},\ i=1,\ldots,d,$ is mapped by $A^{\varepsilon}_1$ to the facet of $C^{\varepsilon}_2$ which is the cone with the edges in the directions $e_{d+l}+e_{d+k}$, $l\in \{1,\ldots,k-1\}$, and $e_i+e_{d+k}$ for $i=1,\ldots,d$, i.e.
the facet of $C^{\varepsilon}_2$ with normal $-\sum_{l=1}^{d+k} e_{l}$. 
Therefore the contactomorphism mentioned above is between an open subset $U^{\varepsilon}$ of $\T^k \times \S^{2d+k-1}$ and the preimage in $\S^{2d+2k-1}$ of the set 
$$C^{\varepsilon}_2 \setminus \{\textrm{ facet with normal }(-\sum_{l=1}^{d+k} e_{l}) \textrm{ and facets with normals }-e_{d+j}, \ j=1,\ldots, k-1\}.$$

Similarly to the case of $\S^1\times \S^{2d}$ analyzed before, we displace certain pre-Lagrangian toric fibers in $M^{\varepsilon}$ by isotopies that are the lifts of isotopies of $\C\P^{d+k-1}$ obtained via the method of probes.
For  $i=1,\ldots, d$ let $F_i$ denote the facet of $\Delta^{d+k-1}(1)$, the moment cone of $(\C\P^{d+k-1}, \frac{1}{\pi} \omega_{FS})$, with normal $-e_i$.
As we have already observed, the vector $e_i$ can be used as a direction of a probe, producing an isotopy of $\C\P^{d+k-1}$ displacing Lagrangian toric fibers corresponding to 
$$x= \frac{1}{2}a_ie_i+\sum_{l\neq i}a_le_l,\textrm{ with }\ \ a_1,\ldots,a_{d+k-1} \geq 0,\ \ \sum_{l=1}^{d+k-1}a_l<1,$$
by an isotopy with support contained in the preimage of $\ii \Delta^{d+k-1}(1) \cup \ii F_i$. 
Lifting this isotopy to the sphere we can displace any pre-Lagrangian toric fiber corresponding to a ray through $y \in C_2^{\varepsilon}\subset \R^{d+k}$ such that
$$y= a_i(2e_{d+k}+e_i)+a_{d+k}e_{d+k}+\sum_{l\neq i,\,d+k}a_l(e_l+e_{d+k}),\textrm{ with }\ \ a_1,\ldots,a_{d+k} \geq 0,$$
with an isotopy supported in the preimage of $\ii C_2^{\varepsilon} \cup \ii \widetilde{F}_i$, where $ \widetilde{F}_i$ is the facet of $C_2^{\varepsilon}$ with normal $-e_i$.
The image of this set under the linear map 
\begin{displaymath}
(A_1^{\varepsilon})^{-1}=\left[
\begin{array}{ccc|cccc}
1 && &&&&\\
&\ddots& &&&&\\
&&1&&&&\\
\hline
-\varepsilon_1&\ldots&-\varepsilon_1&\varepsilon_1&&&\\
-\varepsilon_2&\ldots&-\varepsilon_2&&\ddots&\\
&\ldots&&&&\varepsilon_{k-1}&\\
-2\varepsilon_k&\ldots&-2\varepsilon_k&-\varepsilon_k&\ldots&-\varepsilon_{k}&\varepsilon_{k}
\end{array} \right],
\end{displaymath}
mapping the cone $C_2^{\varepsilon}$ to the cone $C_1^{\varepsilon}$, is the set of points 
$$y= a_i(e_i-\sum_{j=1}^{k-1}\varepsilon_j e_{d+j})+a_{d+k}\varepsilon_{k}e_{d+k}+\sum_{l\in[d]\setminus\{i\}}a_l(e_l-\sum_{j=1}^k\varepsilon_j e_{d+j})+
\sum_{l=1}^{k-1}a_{d+l}(\varepsilon_le_{d+l}),$$
with $a_1,\ldots,a_{d+k}>0$. Here we used the symbol $[d]$ to denote the set $\{1,\ldots,d\}$.
 These are the points $y=(y_1,\ldots,y_{d+k}) \in C_1^{\varepsilon}$ with coordinates
$$y=(a_1,\ \ldots,\ a_d,\ \varepsilon_{1}(a_{d+1}-\sum_{l\in[d]}a_l),\  \ldots,\ \varepsilon_{k-1}(a_{d+k-1}-\sum_{l\in[d]}a_l), \ \varepsilon_{k}(a_{d+k}-\sum_{l\in[d]\setminus\{i\}}a_l)),$$
 that is, the points $y=(y_1,\ldots,y_{d+k}) \in C_1^{\varepsilon}$ with $y_l>0$ for $l\in[d]$, and 
 $$\varepsilon_{j} y_{d+j}+\sum_{l\in[d]}y_l=a_{d+j}>0,\ \ \textrm{ for all }j=1,\ldots,k-1,$$
  $$\varepsilon_{j} y_{d+k}+\sum_{l\in[d]\setminus\{i\}}y_l=a_{d+k}>0.$$
   This set contains, in particular, the set 
 $$(\R_{> 0})^d \times \varepsilon_1\R_{\geq 0} \times \ldots \times  \varepsilon_k\R_{\geq 0},$$
 if $d \geq 2$, and if $d=1$ the set
  $$\R_{> 0} \times \varepsilon_1\R_{\geq 0} \times \ldots \times  \varepsilon_k\R_{> 0}.$$
  Here $\varepsilon_j \R_{\geq 0}$ denotes the set $\{x \in \R\,|\, \varepsilon_j x \geq 0\}$.
 The isotopies are supported only in 
 the preimage of 
 $$\ii C_1^{\varepsilon} \cup \ii \textrm{of the facet with normal }(-e_i)$$
 and thus can be extended to isotopies of the whole $\T^k \times \S^{2d+k-1}$.
 Therefore the corresponding fibers are displaceable in $\T^k \times \S^{2d+k-1}$.
 
 The above procedure can be performed for each $\varepsilon \in \{-1,1\}^k$.
 In this way we can displace all pre-Lagrangian toric fibers  in $\T^k \times \S^{2d+k-1}$ corresponding to the rays in $(\R_{> 0})^d \times \R^k$ if $d\geq 2$,
 and to the rays in $\R_{> 0} \times \R^{k-1} \times \R_{\neq 0}$  if $d=1$.
In particular we displace {\it all} pre-Lagrangian toric fibers in $\T^k \times \S^{2d+k-1}$ for $d\geq 2$,
and {\it all apart from} the fibers corresponding to the rays $\{(y_1,\ldots,y_{k},0)\}$ in 
$\R_{> 0} \times \R^{k}$ if $d=1$. 

Note that while mapping the cone $C_1^{\varepsilon}$ to the cone $C_2^{\varepsilon}$, we made a choice to map the direction of $\varepsilon_{d+k}e_{d+k}$ to $e_{d+k}$. The direction $e_{d+k}$ is special while viewing the sphere as the prequantization of the projective space, as it corresponds to the direction of the fibers of prequantization map (Reeb direction). 
We could have made a different choice and require that $A_1^{\varepsilon}$ maps the direction of $\varepsilon_{d+j}e_{d+j}$ to $e_{d+k}$ for some $j=1,\ldots,k-1$. Repeating the whole procedure with this choice allows us to displace the fibers corresponding to the rays $y=(y_1,\ldots,y_{1+k}) \in \R_{> 0} \times \R^k$ with $y_{j+1} \neq 0$.
Still, this method does not displace the pre-Lagrangian fiber corresponding to the ray $y=(y_1,0,\ldots,0) \in \R_{> 0} \times \R^k$.
Recall that in the case $k=1$ and $d=1$ we have also displaced all the fibers apart from the fiber corresponding to the ray $y=(y_1,0) \in \R_{> 0} \times \R$.
Performing a similar procedure with the use of lens spaces instead of spheres allows us to displace these special pre-Lagrangian toric fibers, as we now explain. 

{\bf Special fibers of the case $d=1$.}
Let $k\geq 1$.
Fix a prime number $p$ and consider the prequantization of $(\C\P^{k}, \frac{p}{\pi}\omega_{FS})$, 
i.e. the lens space $L_p^{2k+1}=\S^{2k+1}/\Z_k$ with the contact structure induced from the standard contact structure on the sphere. This is a contact toric manifold whose contact moment cone, that we denote $C_3$, is the convex cone in $\R^{k+1}$ spanned by the directions
$\eta_l=pe_l+e_{k+1}$, for $l=1,\ldots,k$, and $\eta_{k+1}=e_{k+1}$. Applying the $GL(k+1,\Z)$ transformation 
$(A^{\varepsilon}_1)^{-1}$ for $\varepsilon=(1,\ldots,1)=:{\mathbf{1}}$  we obtain a convex cone, $C_1$, whose edges are in the directions 
$$(A^{\mathbf{1}}_1)^{-1} (pe_1+e_{k+1})=pe_1-\sum_{l=2}^kpe_l +(-2p+1)e_{k+1},\ \ 
(A^{\mathbf{1}}_1)^{-1} (e_{k+1})=e_{k+1},$$
$$(A^{\mathbf{1}}_1)^{-1} (pe_j+e_{k+1})=pe_j+(-p+1)e_{k+1},\ j=2,\ldots,k.$$
The outward normals of this cone are
$$\eta_1:=-e_1,\ \eta_j:=-e_1-e_j, \,\textrm{ for }j=2,\ldots,k,$$
and
$$\eta_{k+1}:=(k-(k+1)p)\,e_1+\left((1-p)\sum^{k-1}_{j=2} e_j\right) -pe_{k+1}=$$
$$\left(k-(k+1)p,\ 1-p,\ \ldots,1-p,\ -p\right).$$
Observe that this is the cone one obtains from $\R_{\geq 0}\times \R^k$, i.e. the moment cone of $\T^k \times \S^{k+1}$, after $k$ contact cuts with respect to the circles generated by the vectors $-\eta_j$ for $j=2,\ldots,k$, and $-\eta_{k+1}$.
Similarly as above one shows that any contact isotopy of $L_p^{2k+1}$ compactly supported in the preimage of 
$$C_1 \setminus \{ \textrm{facets with normals } \eta_2,\ldots,\eta_{k+1}\}$$
$$=(A^{\mathbf{1}}_1)^{-1}(C_3 \setminus \{ \textrm{facets with normals } -e_2,\,\ldots, -e_k, (\sum_{j=1}^k e_j)-p e_{k+1}\})$$
can be extended to an isotopy of $\T^k \times \S^{k+1}$.
Observe that, using the method of probes, for probes in $(\C\P^{k}, \frac{p}{\pi}\omega_{FS})$ supported in the facet with 
the outward normal $-e_1$ one can displace all the fibers corresponding to the points in $C_3$ which are contained in a convex cone, denoted by $C_3^{disp}$, spanned by the directions 
$$pe_1+2e_{k+1},\ pe_j+e_{k+1},\ j=2,\ldots,k,\textrm{ and }e_{k+1}.$$
Note that 
$(A^{\mathbf{1}}_1)^{-1}(C_3^{disp})$ is the convex cone spanned by the directions
$$pe_1-p\,\sum_{l=2}^k\,e_l +(-2p+2)e_{k+1},\  pe_j+(-p+1)e_{k+1},\ j=2,\ldots,k,\textrm{ and } e_{k+1}.$$
Therefore pre-Lagrangian toric fibers corresponding to the rays in $(A^{\varepsilon}_1)^{-1}(C_3^{disp}) \subset \R_{\geq 0}\times \R^k$ are displaceable.
Note that $(A^{\varepsilon}_1)^{-1}(C_3^{disp})$ contains the set $\{(y_1,0,\ldots,0)\}$.
Indeed, given any $(y_1,0,\ldots,0) \in \R_{> 0} \times \R^k$ note that
$$\frac{y_1}{p} (pe_1-p\,\sum_{l=2}^k\,e_l +(-2p+2)e_{k+1})+ \frac{y_1}{p}\,\sum_{j+2}^k\, pe_j+(-p+1)e_{k+1} + 
(k+1)\frac{y_1}{p} (p-1)e_{k+1}$$
$$=(y_1,0,\ldots,0).$$
\end{proof}

\begin{example}
To illustrate this idea better we take a closer look at the case when $k=2$.
Then the cone $C_3$ is spanned by the directions $\eta_1=pe_1+e_3$, 
$\eta_2=pe_2+e_3$, and $\eta_3=e_3$. The matrix 
\begin{displaymath}
(A^{\mathbf{1}}_1)^{-1}=\left(\begin{array}{rrr}
1&0&\,0\,\\
-1&1&\,0\,\\
-2&-1&\,1\,
\end{array}\right)
\end{displaymath}
maps the cone $C_3$ to a cone, called $C_1$, spanned by the directions  
$$(p,-p,-2p+1),\ (0,p,-p+1),\ (0,0,1).$$
The outward normals of $C_1$ are 
$$\eta_1=(-1,0,0),\ \eta_2=(-1,-1,0),\ \eta_3=(2-3p,1-p,-p).$$
One obtains $C_1$ from $\R_{\geq 0}\times \R^2$, i.e. the moment cone of $\T^2 \times \S^{3}$, 
by performing two contact cuts with respect to the circles generated by the elements $-\eta_2=(1,1,0)$ and $-\eta_3=(3p-2,p-1,p)$ in the Lie algebra of $\T^3$.
Any contact isotopy of $L_p^{5}$ compactly supported in the preimage of 
$$C_1 \setminus \{ \textrm{facets with normals } (-1,-1,0),\ (2-3p,1-p,-p)\}$$
$$=(A^{\mathbf{1}}_1)^{-1}(C_3 \setminus \{ \textrm{facets with normals } (0,-1,0),\ (1,1,-p)\})$$
can be extended to an isotopy of $\T^2 \times \S^{3}$.
Viewing $L_p^{5}$ as the prequantization of $(\C\P^{2}, \frac{p}{\pi}\omega_{FS})$
one construct isotopies, with compact supports contained in the above sets, displacing pre-Lagrangian toric fibers of $L_p^{5}$
corresponding to the rays in the interior of the cone $C_3^{disp}$ 
spanned by the directions
 $(p,0,2),$ $(0,p,1),$ and $(0,0,1).$
Then the set $(A^{\mathbf{1}}_1)^{-1}(C_3^{disp})$ is a cone spanned by the directions 
$(p,-p,-2p+2)$, $(0,p,-p+1)$, and $(0,0,1)$ 
and it contains the set $\R_{> 0} \times \{0\}\times \{0\}$
as given any $(y,0,0)$ we have that
$$\frac y p (p,-p,-2p+2)+\frac y p (0,p,-p+1)+\frac{3y}{ p}(3p-3)(0,0,1)=(y,0,0).$$
\end{example}

\section{Free torus actions and non-displaceability of pre-Lagrangian toric fibers.}\label{section free non-displaceable}
In this section we prove that in all contact toric manifolds with free toric action, except possibly for non-trivial principal $\T^3$-bundles over $\S^2,$ all pre-Lagrangian toric fibers are non-displaceable. Moreover we show that all contact toric manifolds whose toric action is not free have uncountably many displaceable pre-Lagrangian toric fibers. Note that when the toric action is free then every orbit is a pre-Lagrangian.
From Lerman's classification of contact toric manifolds, recalled here on page \pageref{classification}, it follows that the contact toric manifolds equipped with a free toric action are $(\T^3,\ker \alpha_k)$, $k\geq 1$, and principal $\T^{d}$ bundles over $\S^{d-1}$ with the unique $\T^d$-invariant contact structures.
Note that if $d \neq 3$ such bundle must be trivial. Indeed, principal $\T^d$-bundles over a manifold are classified by homotopy classes of maps from this manifold to the classifying space $B\T^d=(\bb{C}\P^{\infty})^d$. As $(\C\P^{\infty})^d$ is an Eilenberg-MacLane space $K(\bb{Z}^d,2)$, there is a bijection between the set of these homotopy classes of maps and the second cohomology group of the given manifold with coefficients in $\bb{Z}^d$. Therefore principal $\T^d$ bundles over $\S^{d-1}$ are classified by $H^2(\S^{d-1};\bb{Z}^d)$. This group is trivial for $d\neq 3$ and equal to $\bb{Z}^3$ for $d=3$.
Therefore
compact connected contact toric manifolds with a free toric action are:
\begin{itemize}
\item In dimension $3$: $(\T^3,\ker \alpha_k),$ $k \geq1$.
\item In dimension $5$: a $\Z^3$ collection of principal $\T^{3}$-bundles over $\S^{2}$, each with a unique $\T^3$-invariant contact structure.
\item In dimension greater than $5$:  $\T^{d}\times\S^{d-1},$ $d\geq4$, with a unique $\T^d$-invariant contact structure.
\end{itemize}

Below we show that all pre-Lagrangian toric fibers in $(\T^3,\ker \alpha_k)$ and in the trivial principal bundles $\T^{d}\times \S^{d-1}$, with $d\geq 3$, i.e. in the cosphere bundles of tori, are non-displaceable. So far we were unable to prove non-displaceability of orbits in non-trivial principal $\T^{3}$ bundles over $\S^{2}$. These contact toric structures are explicitely described in \cite{marinkovic} where it is shown that each of them is contactomorphic to one of $\T^2\times L_k,$ $k\in\mathbb{N},$ ($L_1=\S^3$) with the unique contact toric structure (up to a reparametrization of the torus) and that the contact toric structure on $\T^2\times L_k=\T^2\times( \S^3/\Z_k)$ is induced by the diagonal $\Z_k$-action on $\S^3$-factor of $\T^2\times\S^3.$ Thus, in order to prove that all pre-Lagrangian toric fibers in $\T^2\times L_k,$ $k>1,$ are non-displaceable it would be enough to prove that all pre-Lagrangian toric fibers in $\T^2\times\S^3$ are non-displaceable.

\subsection{Non-displaceability in $(\T^3,\ker \alpha_k)$.}
In this section we prove Proposition \ref{non-displaceable 3torus}, i.e. the fact that every pre-Lagrangian toric orbit in the contact toric manifold 
  $(\T^3=\S^1_{(\theta)}\times \T^2_{(\theta_1,\theta_2)}, \ker \alpha_k),$ with $k\geq 1$, where $\alpha_k=\cos (2\pi k \theta)\,d\theta_1+\sin (2\pi k \theta)\,d\theta_2$, is non-displaceable.
We are grateful to Patrick Massot and the anonymous referee for explaining to us 
how the proof follows from the classification of tight contact structures on $\T^2\times[0,1].$

The toric $\T^2$-action on $(\T^3,\ker\alpha_k)$ is given by $(t_1,t_2)\ast(\theta,\theta_1,\theta_2)=(\theta,\theta_1+t_1,\theta_2+t_2).$ The moment map with respect to the contact form $\alpha_k$ is given by $\mu_k(\theta,\theta_1,\theta_2)=(\cos (2\pi k \theta),\sin (2\pi k \theta)).$ Therefore each fiber of $\mu_k$ has $k$ connected components. The manifolds $(\T^3,\ker\alpha_k)$ with $k>1$ are the only contact toric manifolds with disconnected moment map fibers.
\begin{proof}[Proof of Proposition \ref{non-displaceable 3torus}]
Let $L_{\theta}=\{\theta\} \times \T^2$ be a toric orbit, i.e. a connected component of some pre-Lagrangian toric fiber in $(\T^3,\ker \alpha_k)$. 
Suppose that
there exists a contact isotopy $\{\phi_t\}_{t\in [0,1]}$ such that 
$\phi_0$ is the identity and $\phi_1(L_{\theta})\cap L_{\theta}=\emptyset$.
Then one can lift this isotopy to an isotopy of $(\R \times \T^2, \ker  (\cos (2\pi \theta)\,d\theta_1 +\sin (2\pi \theta) \, d\theta_2))$ using the covering map 
$\R  \times \T^2 \rightarrow \S^1 \times \T^2$ given by $(\theta,\theta_1,\theta_2)\mapsto (\frac 1 k \theta,\theta_1,\theta_2)$, which is a local contactomorphism.
Let $\{\Phi_t\}_{t\in [0,1]}$ denote the lift of $\{\phi_t\}_{t\in [0,1]}$ to an isotopy of $\R \times \T^2$.
Then $\Phi_t$ displaces $L=\{k\theta \} \times \T^2$.
As $L$ is compact, one can construct a smooth ``cut-off" function, i.e. a function with compact support, which is $1$ on a neighborhood of $\cup_{t\in [0,1]} \Phi_t(L)$. 
Multiplying the contact Hamiltonian generating $\Phi_t$ by this cut-off function, we can modify $\Phi_t$ to an isotopy with compact support contained in $(k\theta-a ,k\theta+a)\times \T^2$ for some integer $a>0$. 
By a slight abuse of notation we still call the resulting isotopy $\Phi_t$.

Let $N$ denote the region in $[k\theta-a ,k\theta+a] \times \T^2$ which is bounded by the disjoint pre-Lagrangian $2$-tori $L$ and $L_1:=\Phi_1(L)$. 
The two boundary components of $N$ 
are boundary parallel in the (irreducible) manifold $[k\theta-a ,k\theta+a] \times \T^2$, and therefore $N$ is diffeomorphic to $[0,1] \times \T^2$. 

Let $\xi_N$ denote the contact structure on $N$ induced by the universally tight contact structure of the ambient space, $\R \times \T^2$.
(The universal cover of $\R \times \T^2$ is $(\R^3,  \ker (\cos (2\pi \theta)\,d\theta_1 +\sin (2\pi \theta) \, d\theta_2))$, which is contactomorphic to $\R^3$ with the standard contact structure and thus is tight by a result of Bennequin.)
Thus the contact structure $\xi_N$ on $N$ is also universally tight, and $(N,\xi_N)$ is contactomorphic to 
one of the structures in Giroux's classification (Theorem 1.5 \cite{Giroux2}; 
We remark here that a classification of tight contact structures on a thickened torus was also obtained independently by Honda, \cite{Honda}).
The characteristic foliations on both (pre-Lagrangian) boundary components of $N$ are linear foliations of the tori, 
and both with the same direction. By results in \cite{Giroux2} such contact manifolds are classified by their $\pi$-torsion as we explain in the next paragraph.

In \cite{Giroux2} Giroux analyses the set $\mathcal{SCT}(I \times \T^2)$
  of tight contact structures on $I \times \T^2$ with given characteristic foliations  
  on the boundary. The connected components of this set are,
by a theorem of Gray, the isotopy classes (for isotopies fixing the boundary) of contact structures on $I \times \T^2$. 
Theorem 1.5 in \cite{Giroux2} states in particular that if
the characteristic foliation on the boundary is topologically linearizable\footnote{The characteristic foliation on the boundary $\T^2$ is called topologically linearizable if all the orbits are dense or all the orbits are closed.}
then the map 
$(\chi_{\partial}, \tau_{\pi})$ from $\mathcal{SCT}(I \times \T^2)$
to $H_1(I \times \T^2;\Z) \times \mathbb{N}$, sending a tight contact structure to the pair (its relative Euler class, its $\pi$-torsion), has connected fibers. 
Moreover the theorem specifies the image of $(\chi_{\partial}, \tau_{\pi})$ on universally tight contact structures. 
For universally tight contact structures with identical linear characteristic foliations on the boundary the image is 
$\{0\} \times \mathbb{N}\subset H_1(I \times \T^2;\Z) \times \mathbb{N}$. Indeed, foliations on both boundary components generate the same element $\sigma \in H_1(\T^2 ;\R)$. Thus, using the notation from  \cite{Giroux2}, $B=\{\sigma\}$ and $X_u=\{0\}$. Therefore in this case the relative Euler class vanishes and the (isotopy classes of) contact structures are classified by their $\pi$-torsion.

Therefore the contact manifold $(N,\xi_N)$ is contactomorphic to
$[k\theta,k\theta+b] \times \T^2$ with the contact structure given by $\ker (\cos (2\pi \theta)\,d\theta_1 +\sin (2\pi \theta) \, d\theta_2)$,
for some $b \in \mathbb{Z}_{>0}$.
It follows that the isotopy $\Phi_t$ induces a contactomorphism from 
$([k\theta-a,k\theta]\times \T^2,  \ker (\cos (2\pi \theta)\,d\theta_1 +\sin (2\pi \theta) \, d\theta_2))$ to 
$([k\theta-a,k\theta+b]\times \T^2,  \ker (\cos (2\pi \theta)\,d\theta_1 +\sin (2\pi \theta) \, d\theta_2))$.
These manifolds have different $\pi$-torsions (by Proposition 3.42 of  \cite{Giroux2}) and thus such a contactomorphism
contradicts the aforementioned
classification of tight contact structures on $[0,1]\times \T^2 $ (Theorem 1.5 \cite{Giroux2}).
This proves that $L_{\theta}$ must be non-displaceable.
\end{proof}
\subsection{Non-displaceability in cosphere bundles of tori.}
Let $N$ be a smooth manifold and $P_+T^*N$ the corresponding cosphere bundle. The Liouville form on $T^*N$ descends to a contact form on $P_+T^*N$ making it a contact manifold.
 Its symplectization $S(P_+T^*N)$ is $T^*N$ without the zero section.
 Since the graph of any closed 1-form on $N$ is a Lagrangian submanifold of $T^*N,$ it follows that 
  the graph $\widetilde{L}_{\alpha}$ of a nowhere vanishing closed 1-form $\alpha$ is a Lagrangian submanifold in the symplectization. If $\pi:S(P_+T^*N)\rightarrow P_+T^*N$ is the projection then $L_{\alpha}=\pi(\widetilde{L}_{\alpha})$ is a pre-Lagrangian submanifold in $P_+T^*N$. 
 Thus to every nowhere vanishing closed 1-form on $N$ corresponds a pre-Lagrangian submanifold in $P_+T^*N.$
 If $N$ is a torus $\T^d$, one obtains as its cosphere bundle the contact manifold
$$P_+T^*\T^d\cong \T^d\times\mathbb{S}^{d-1}=\{(e^{2\pi i\theta_1}\ldots,e^{2\pi i\theta_d},x_1,\ldots,x_d)\in \T^d\times\mathbb{R}^d\hskip1mm|\hskip1mm \sum_{i=1}^dx_i^2=1\}$$
with the contact structure given by the kernel of the $1$-form $\sum_{i=1}^dx_id\theta_i$.
The $\T^d$-action on $(\T^d\times\mathbb{S}^{d-1},\ker(\sum_{i=1}^dx_id\theta_i))$ defined by
$$(t_1,\ldots, t_d)\ast(e^{2\pi i\theta_1},\ldots,e^{2\pi i\theta_d},x_1,\ldots,x_d)\rightarrow(t_1e^{2\pi i\theta_1},\ldots,t_d e^{2\pi i\theta_d},x_1,\ldots,x_d),$$ 
is an effective action by contactomorphisms and thus 
 it turns $\T^d\times\mathbb{S}^{d-1}$ into a contact toric manifold. 
 The moment map, $\mu:\T^d\times\mathbb{S}^{d-1}\rightarrow(\mathbb{R}^d)^*$, associated to the contact form $\sum_{i=1}^dx_id\theta_i$ is given by
 $$\mu(e^{2\pi i\theta_1}\ldots,e^{2\pi i\theta_d},x_1,\ldots,x_d)=\pi(x_1,\ldots,x_d).$$
 Hence, every pre-Lagrangian toric fiber in $\T^d\times\mathbb{S}^{d-1}$ is of the form 
 $\T^d\times\{\textrm{p}\}, $ for some point $\textrm{p}\in\mathbb{S}^{d-1}.$
 All these orbits are non-displaceable as we now explain.

 \begin{proof}[Proof of Proposition \ref{non-displaceable cosphere}]
 Take any pre-Lagrangian toric fiber $\T^d \times \{p\}$, $p=(p_1,\ldots,p_d) \in \mathbb{S}^{d-1}$,
 and observe that $\T^d \times \{p\}$ is the pre-Lagrangian corresponding to the graph of 
 the nowhere vanishing closed $1$-form $\sum_{i=1}^dp_id\theta_i$ on $\T^d$. 
 If all $p_j$'s are rational then $\T^d \times \{p\}$ is foliated by $(d-1)$-dimensional Legendrians. 
 Such Legendrians cannot be displaced from $\T^d \times \{p\}$: indeed,
 Corollary 2.5.2 from \cite{EliashbergHoferSalamon} says that if $\Lambda$ is any of these Legendrians and $\{\varphi_t\}$ is any contact isotopy, with $\varphi_0=id$ and $\varphi_1(\Lambda)$ transverse to $\T^d \times \{p\}$, then
 the number of intersection points $\varphi_1(\Lambda) \cap (\T^d \times \{p\})$
  is at least $2^{d-1}.$ In particular, $\T^d \times \{p\}$ is non-displaceable.
  As the set of points $p$ with rational coefficients is dense in $\S^{d-1}$, and displaceability is an open property, we deduce that all pre-Lagrangian toric fibers $\T^d \times \{p\}$ are non-displaceable.
Alternatively, instead of \cite{EliashbergHoferSalamon}, one could use Example 2.4.A from \cite{elpo} showing that these fibers are not only non-displaceable, but also stably non-displaceable.
 \end{proof}
 \subsection{Displaceability of pre-Lagrangian toric fibers when a toric action is not free.}\label{subsection not free displace}
 The goal of this Section is to prove Theorem \ref{not free displaceable}, i.e. to show that
 every compact connected contact toric manifold for which the toric action is not free contains uncountably many displaceable pre-Lagrangian toric fibers.
 The idea of the proof is the following.
 We consider separately manifolds of Reeb and not of Reeb type. 
  We show that the first ones are prequantizations of symplectic toric orbifolds (Lemma \ref{lemma prequantization of orbifold} below), and thus one can displace their fibers using the methods of Section \ref{section prequantization} (Corollary \ref{of reeb type displace} below). 
 The second ones are either overtwisted contact toric manifolds of dimension 3 or $\T^k\times\S^{2d+k-1}$ with $d\geq1,$ $k\geq1$. In Proposition \ref{3dim displace} below we analyze the $3$-dimensional overtwisted case and displace uncountably many fibers using the method of contact cuts. All the fibers of 
 $\T^k\times\S^{2d+k-1},$ $d\geq1,$ $k\geq1$, 
are displaceable by Theorem \ref{displace in boundary of stabilization}, already proved in Section \ref{section displace by cuts}.
  \subsubsection{Contact toric manifolds of Reeb type.}\label{section reeb type}
Let $(V^{2d-1},\xi)$ be a contact toric manifold and $\alpha$ a $\T^d$-invariant contact form for $\xi$, giving rise to the $\alpha$-moment map $\mu_{\alpha}.$ Then 
 the flow of the Reeb vector field $R_{\alpha}$ preserves the level sets of $\mu_{\alpha}$, as for any point $p\in V$ and any vector
 $X\in Lie(\T^d)$, we have
$$\langle d\mu_{\alpha}(R_{\alpha}(p)),X\rangle=d\alpha(R_{\alpha}(p),X)=0,$$
implying that 
$d\mu_{\alpha}(R_{\alpha}(p))=0$ for all $p \in V$ (Lemma 7.7.4 \cite{Geiges}).
Thus  the Reeb orbit is contained in the $\T^d$-orbit.
If in addition the Reeb vector field corresponds to some $\eta \in \mathfrak{t}^d$ in the Lie algebra of $\T^d$ then $V$ is called a contact toric manifold {\bf of Reeb type}. For these manifolds one can always perturb a contact form (keeping same contact structure) so that an integral multiple of the associated Reeb vector field corresponds to a lattice element $\eta \in \mathfrak{t}_{\bb{Z}}^d$  and thus the Reeb flow generates an $\S^1$-action which is a subaction of the action of $\T^d$ (\cite{BoyerGalicki}).
In that case the image of the contact moment map is a good strictly convex cone (Definition 2.17 in \cite{Lerman}), namely a cone over a convex polytope (Theorem 4.9 in \cite{BoyerGalicki}). Moreover, such contact manifolds can be obtained as contact reductions of the standard contact sphere (Theorem 5.1 in \cite{BoyerGalicki}). 

Let $V$ be any contact toric manifold of Reeb type and $C$ the corresponding good strictly convex cone.
Denote by
  $v_1,\ldots,v_N\in\mathbb{Z}^d$ the primitive inward vectors normal to the facets of the cone $C.$ That is
$$C=\bigcap_{i=1}^N\{x=(x_1,\ldots,x_d)\in (\mathbb{R}^d)^*| \langle x,v_i\rangle\geq0\}.$$
For any $R=\sum_{i=1}^Na_iv_i$ with $a_1,\ldots,a_N  \in \bb{R}_{>0}$ there exists a $\T^d$-invariant contact form $\alpha$ such that $R=R_{\alpha} $ is the Reeb vector field corresponding to it (Proposition 2.19 in \cite{abreumacarinicontact}). Take any such $R$ in the lattice of the Lie algebra of $T$, i.e. $R=\sum_{i=1}^Na_iv_i \in \mathfrak{t}^d_{\bb{Z}} $, $a_1,\ldots,a_N  \in \bb{R}_{>0}$.
The corresponding Reeb flow generates an $\S^1$-action on $V$ and one can perform symplectic reduction of the symplectization $SV$ with respect to the lift of that $\S^1$-action. The result of this reduction is a symplectic orbifold $M$ and $V$ is a prequantization of $M$ (Theorem 2.7 in \cite{Boyer}; see also Lemma 3.7 in \cite{Lerman2}).
Recall that symplectic toric orbifolds are classified by rational and simple polytopes with positive integral labels attached to each facet (\cite{LermanTolman}). The points in the preimage of a facet with label $m$ have neighborhoods modeled on $\bb{C}^d /\bb{Z}_m$.
The polytope corresponding to $M$ is the intersection of the moment cone $C$ for $V$ with the hyperplane perpendicular to $R$.
Note that this intersection is always a compact polytope. Indeed, if the hyperplane is given by $h:=\{x \in (\bb{R}^d)^*\,|\,\langle x,R\rangle =c\}$ then any $x \in C \cap h$ must satisfy
$$a_1\langle x,v_1\rangle+\cdots+a_N\langle x,v_N\rangle=c,\,\,\langle x,v_1\rangle\geq0,\ldots, \langle x,v_N\rangle\geq0.$$
As $a_j >0$ and $v_1,\ldots,v_N$ generate the whole $\bb{R}^d$ (due to strict convexity), the above condition implies that $0 \leq \langle x,v_j\rangle\leq\frac{c}{a_j}$ and $C \cap h$ is compact (it is empty if $c<0$).
The facets of the polytope of $M$ are intersections of the facets of $C$ with the hyperplane $h$. The label of a facet corresponding to the facet of $C$ with inward normal $v_j$ is the index of the lattice generated by $v_j$ and $R$ inside $span_{\bb{R}}( v_j,R) \cap \bb{Z}^d$, where $\bb{Z}^d$ is the lattice of $\bb{R}^d=\mathfrak{t}^d$. If the collection $\{v_j,R\}$ can be completed to a $\bb{Z}$ basis of $\bb{Z}^d$, the label on the corresponding facet is $1$.
\begin{lemma} \label{lemma prequantization of orbifold}
If $V$ is a contact toric manifold of Reeb type then $V$ can be presented as a prequantization of some symplectic toric orbifold having at least one label equal to $1.$
\end{lemma}
\begin{proof}
Let C be a good, strictly convex cone corresponding to $V$, with inward normals $v_1,\ldots,v_N \in \bb{R}^d$. Suppose the facets corresponding to normals  $v_1,\ldots,v_{d-1}$ intersect in an edge.
The assumption that the cone is good implies that the vectors $v_1,\ldots,v_{d-1}$ form a $\bb{Z}$-basis of $\bb{Z}^{d-1}=span_{\bb{R}}(v_1,\ldots,v_{d-1}) \, \cap\, \mathfrak{t}_{\bb{Z}}$. Without loss of generality we can assume that $v_1=(1,0,\ldots, 0)$, ... , $v_{d-1}=(0,\ldots, 0, 1, 0)$. We show that one can choose $a_j >0$, $j=1,\ldots,N$, so that $R=\sum_{j=1}^{N}\, a_j v_j =(c_1,\ldots, c_{d-1},1)$ or $R=(c_1,\ldots, c_{d-1},-1)$ with each $c_j \in \bb{Z}$. 
In both cases the collection $\{v_1,\ldots,v_{d-1}, R\}$ forms a $\bb{Z}$-basis of $\bb{Z}^d$. This implies that $V$ is a prequantization of a symplectic toric orbifold obtained as a reduction of $SV$ by the circle action generated by $R$, and that facets of this orbifold corresponding to normals $v_1, \ldots, v_{d-1}$ have labels $1$.
\vskip1mm
For each $v_j$ let $(v_{j,1}, \ldots, v_{j,d})$ denote its coordinates. Suppose that at least one of $v_{d,d}, \ldots, v_{N,d}$ is positive. With this assumption, we find constants $a_j>0$ such that $R=\sum_{j=1}^{N}\, a_j v_j =(c_1,\ldots, c_{d-1},1)$.
(If all $v_{d,d}, \ldots, v_{N,d}$ were non positive
we would similarly find $a_j >0$, so that $R=(c_1,\ldots, c_{d-1},-1)\in \bb{Z}^d$.) 
Note that at least one of $v_{d,d}, \ldots, v_{N,d}$ is non-zero as $v_1,\ldots, v_N$ form a basis of $\bb{R}^d.$
Take any $a_{d}, \ldots, a_N \in \bb{R}_+$ such that $\sum_{j=d}^{N} a_j \,v_{j,d}=1$ (so also $\sum_{j=1}^{N} a_j \,v_{j,d}=1$). For $a_j$, $j=1,\ldots, d-1$, take any positive real number such that $a_j + \sum_{l=d}^{N} a_l \,v_{l,j}$ is an integer. For example one can take $ a_j = -\sum_{l=d}^{N} a_l \,v_{l,j}$ if this sum was non-positive, and $a_j = \left \lceil{\sum_{l=d}^{N} a_l \,v_{l,j}}\right \rceil - \sum_{l=d}^{N} a_l \,v_{l,j}$ if this sum is positive.
In both cases $\sum_{l=1}^{N}v_{l,j}= a_j + \sum_{l=d}^{N} a_l \,v_{l,j}$ is an integer.
\end{proof}
\begin{corollary}\label{of reeb type displace} Any contact toric manifold $V$ of Reeb type contains uncountably many displaceable pre-Lagrangian toric fibers.
\end{corollary}
\begin{proof} The proof uses McDuff's method of probes for displacing Lagrangian fibers in symplectic toric manifolds (\cite{mcduff}, \cite{abreubormanmcduff}), recalled here in Section \ref{method of probes}. Though this method was developed for symplectic toric manifolds, it can be generalized to symplectic toric orbifolds as long as the probe starts at a facet with label $1$. This is because the isotopy constructed via the method of probes is supported only in a small neighborhood of the preimage of a probe. A probe is a half open interval almost entirely contained in the interior of a moment polytope: only the starting point of a probe lies on the interior of a facet. Therefore if only this facet has label $1$, a small enough neighborhood of the preimage of a probe contains no orbifold points. Using Lemma \ref{lemma prequantization of orbifold}, the method of probes and Proposition \ref{lift.preq} we can displace uncountably many pre-Lagrangian toric fibers.
\end{proof}

\subsubsection{Toric contact manifolds not of Reeb type, with toric action that is not free.} 
{\bf Dimension $3$.}
\\We start by analyzing $3$-dimensional manifolds separately as these are not, in general, uniquely determined by their moment cones.
\begin{proposition}\label{3dim displace}
Let $V$ be a $3$-dimensional contact toric manifold not of Reeb type and such that the toric action is not free. Then $V$ contains uncountably many displaceable pre-Lagrangian toric fibers.
\end{proposition}
\begin{proof}
It was shown by Lerman in Section 6.2 of \cite{Lerman} (proof of Theorem 2.18 (2)) that 
$3$-dimensional contact toric manifold for which the toric action is not free
are diffeomorphic to lens spaces (including $\S^3$ and $\S^1 \times \S^2$), and are classified by pairs of real numbers $(t_1,t_2)$ such that
$0\leq t_1 <2\pi,$ $t_1< t_2,$ and $\tan t_i \in \Q$ or $\cos t_i=0$, for each $i=1,\,2$. The last condition means that $(\cos t_i,\sin t_i) \in \R^2$ lies on a ray through some $(m_i,n_i)\in \Z^2$.
Moreover, the assumption of being not of Reeb type implies that the moment cone is not strictly convex
and thus $t_1+\pi \leq t_2$.
The manifold corresponding to $(t_1,t_2)$, denoted here by $M_{t_1,t_2}$, is obtained from $\R \times \S^1\times \S^1$  with the contact form 
$\alpha=\cos t \,d\theta_1+ \sin t\,d\theta_2$ at $(t,\theta_1,\theta_2) \in \R \times \S^1\times \S^1$ by contact cutting 
performed twice. 
Namely, if $(m_i,n_i)\in \Z^2$ denotes the primitive integral vector in the direction $(\cos t_i,\sin t_i)$, $i=1,2$, then
$$M_{t_1,t_2}=[t_1,t_2] \times \S^1\times \S^1/\sim,$$
where the relation $\sim$ denotes the following identifications on the boundary of $[t_1,t_2] \times \S^1\times \S^1$:
on $\{t_1\} \times \S^1\times \S^1$ we identify the orbits of a circle action generated by 
$-n_1 \frac{\partial}{\partial \theta_1}+m_1 \frac{\partial}{\partial \theta_2}$,
and 
on $\{t_2\} \times \S^1\times \S^1$ we identify the orbits of a circle action generated by 
$n_2 \frac{\partial}{\partial \theta_1}-m_2 \frac{\partial}{\partial \theta_2}$.
Topologically this is equivalent to gluing two solid tori along their boundaries by various automorphisms of the boundary therefore $M_{t_1,t_2}$ is a lens space. 
The moment cone of $M_{t_1,t_2}$ is the cone in $\R^2$ over the image of the interval $[t_1,t_2]\subset \R$ under the map $t\mapsto (\cos t, \sin t) \in \R^2$. Hence it is the whole $\R^2$ whenever $t_2-t_1 \geq 2\pi$.

Take any $t_2'>t_1$ such that $t_2'-t_1<\pi$ and $(\cos t_2',\sin t_2') \in \R^2$ lies on a ray through some $(m_2',n_2')\in \Z^2$.
It is obvious from the construction that $M_{t_1,t_2'}$ can be obtained from $M_{t_1,t_2}$ by a contact cut. Therefore any isotopy of $M_{t_1,t_2'}$ whose support is compact and contained in 
$([t_1,t_2') \times \S^1\times \S^1/\sim )\ \subset M_{t_1,t_2'}$, can be extended to an isotopy of $M_{t_1,t_2}$.
Observe that the moment cone of $M_{t_1,t_2'}$ is strictly convex, thus $M_{t_1,t_2'}$ is a prequantization space. 
Similarly to what was done in Section \ref{section displace by cuts}, one can displace uncountably many pre-Lagrangian toric fibers of $M_{t_1,t_2'}$ by isotopies which are lifts of isotopies obtained via the probes method.
If a probe is based on a facet ``far from the cut", i.e. such that it's preimage under prequantization map is disjoint from the set 
$\{t_2'\}\times \S^1\times \S^1/\sim$, then such isotopy has a compact support contained in $([t_1,t_2') \times \S^1\times \S^1/\sim )$ and can be extended to an isotopy of $M_{t_1,t_2}$.
\end{proof}

\noindent {\bf Dimension greater than $3$.}
\\We now analyze the contact toric manifolds not of Reeb type, whose action is not free, and whose dimension is greater than $3$. 
These are determined by their (not strictly convex) good moment cones. Let $k$ be the dimension of the maximal linear subspace contained in the good moment cone, and $(d+k)$, $d\geq  1$, be the dimension of the torus acting. 
(When $d=0$ the manifold is necessarily a $\T^k$ bundle over $\S^{k-1}$ and the toric action is free).
Then such a cone is isomorphic to the cone of $\T^k \times \S^{2d+k-1}$ and therefore the manifold must be equivariantly contactomorphic to $\T^k \times \S^{2d+k-1}$ with the unique $\T^{d+k}$ invariant contact structure, described in Section \ref{section displace by cuts}, (Theorem 2.18 in \cite{Lerman}; see also the proof of Theorem 1.3 in Section 7 of \cite{Lerman}). 
Thus, by Theorem \ref{displace in boundary of stabilization} (Section \ref{section displace by cuts}) all pre-Lagrangian toric fibers in $\T^k \times \S^{2d+k-1}$ are displaceable. 
This concludes the proof of Theorem \ref{not free displaceable}.
\section*{Acknowledgments}
The authors are very grateful to Patrick Massot and to the anonymous referee who independently 
pointed out to us that the proof in Remark \ref{displace all with one iso} can be simplified to the proof of Proposition \ref{displace sphere},
and they explained to us how to prove Proposition \ref{non-displaceable 3torus}.
Additionally we would like to thank the referee for his/her comments that helped us improve the exposition.
Moreover we thank Miguel Abreu and Strom Borman for helpful comments on the first version of the paper and Sheila Sandon, Maia Fraser and Roger Casals for useful discussions. 
The authors were supported by the Funda\c{c}\~ao para a Ci\^encia e a Tecnologia (FCT, Portugal):
fellowships SFRH/BD/77639/2011 (Marinkovi\'c) and SFRH/BPD/87791/2012 (Pabiniak);
projects PTDC/MAT/117762/2010 (Marinkovi\'c and Pabiniak) and 
EXCL/MAT-GEO/0222/2012 (Pabiniak).


\begin{thebibliography}{99}
\bibitem{abreubormanmcduff} M. Abreu, M. S. Borman, and D. McDuff, Displacing Lagrangian toric fibers by
extended probes, \emph{Algebr. Geom. Topol}. 14 (2014), 687--752.
\bibitem{abreumacarini} M. Abreu and L. Macarini, Remarks on Lagrangian intersections in toric manifolds, \emph{Trans. Amer. Math. Soc.}
365 (2013), 3851--3875.
\bibitem{abreumacarinicontact} M. Abreu and L. Macarini, Contact homology of good toric contact manifolds, \emph{Compos. Math.} 148 (2012), 304--334.
\bibitem{Albert} C. Albert, Le th\'eor\`eme de r\'eduction de Marsden-Weinstein en g\'eom\'etrie cosymplectique et de contact, \emph{J. Geom. Phys.} 6 (1989), 627--649.
\bibitem{BanyagaMolino1} A. Banyaga and P. Molino, G\'eom\'etrie des formes de contact compl\'etement int\'egrables de type toriques, \emph{ S\'eminaire Gaston Darboux de G\'eom\'etrie et Topologie Diff\'erentielle}, 1991-1992 (Montpellier), Univ. Montpellier II, Montpellier, (1993), 1--25.
\bibitem{BanyagaMolino2} A. Banyaga and P. Molino, Complete integrability in contact geometry, Penn State preprint PM 197, (1996).
\bibitem{Bhupal} M. Bhupal, A partial order on the group of contactomorphisms of $\R^{2n+1}$ via generating functions, \emph{Turkish J. Math}., 25 (2001), 125--135.
\bibitem{BoothbyWang} W.M. Boothby and H.C. Wang, On contact manifolds, \emph{Ann. of Math.} (2) 68 (1958), 721--734.
\bibitem{BormanZapolsky} M. S. Borman and Zapolsky, Quasimorphisms on contactomorphism groups and contact rigidity, \emph{Geom. Topol.} 19 (2015), 365--411.
\bibitem{Boyer} C.P. Boyer, Maximal tori in contactomorphism groups, \emph{Diff. Geom. Appl.} 31 (2013),  190--216.
\bibitem{BoyerGalicki} C.P. Boyer and K. Galicki, A note on toric contact geometry, \emph{J. Geom. Phys}. 35 (2000), 288--298.
\bibitem{ChoPoddar}
C.-H. Cho and M. Poddar, Holomorphic orbidiscs and Lagrangian Floer cohomology of compact toric
orbifolds, \emph{J. Diff. Geom} 98 (2014), 21-116
\bibitem{Delzant} Th. Delzant, Hamiltoniens P\'eriodiques et Images Convexes de L'application Moment, \emph{Bull. Soc. Math. France }116 (1988), 315--339.
\bibitem{EliashbergHoferSalamon} Y. Eliashberg, H. Hofer and D. Salamon, Lagrangian intersection in contact geometry, \emph{Geom. Funct. Anal.} 5 (1995), 244-269.
\bibitem{EKimP} Y. Eliashberg, S. S. Kim, and L. Polterovich, Geometry of contact transformations and
domains: orderability versus squeezing, \emph{Geom. Topol.} 10 (2006), 635--1747.
\bibitem{elpo} Y. Eliashberg and L. Polterovich, Partially ordered groups and geometry of contact transformations, \emph{Geom. Funct. Anal.} 10 (2000), 1448--1476.
\bibitem{entpol} M. Entov and L. Polterovich, Quasi-states and symplectic intersections, \emph{Comment.
Math. Helv.}, 81 (2006), 75--99.

\bibitem{fooo:2010} K. Fukaya, Y.-G. Oh, H. Ohta and K. Ono, Lagrangian Floer theory on compact toric manifolds I, \emph{Duke Math. J}.151 (2010), 23-174.

\bibitem{fooo:2011} K. Fukaya, Y.-G. Oh, H. Ohta and K. Ono,  Lagrangian Floer theory on compact toric manifolds II: bulk deformations, \emph{Selecta Math.} 17 (2011), 609-711.

\bibitem{FPR} M. Fraser, L. Polterovich and D. Rosen, On Sandon-type metrics for contactomorphism groups, arXiv:1207.3151.

\bibitem{Geiges} H. Geiges, An Introduction to Contact Topology, \emph{Cambridge University Press}, 2008.

\bibitem{Giroux2} E. Giroux, Structures de contact en dimension trois et bifurcations des feuilletages de surfaces, \emph{Invent. Math.} 141 (2000), 615--689.

\bibitem{Giroux} E. Giroux, Sur la  g\'eom\'etrie et la dynamique des transformations de contact
[d'apr\'es Y. Eliashberg, L. Polterovich et al.] S\'eminaire Bourbaki
61\'eme ann\'ee, 2008--2009, no 1004.

\bibitem{GW} E. Gonzales, C. Woodward Quantum Cohomology and toric minimal model program. arXiv:1207.3253v5.

\bibitem{GKPS} G. Granja, Y. Karshon, M. Pabiniak, S. Sandon, Non-linear Maslov index for lens spaces, \emph{in preparation.}

\bibitem{Honda} K. Honda, On the classification of tight contact structures I, \emph{Geom. Topol.} 4 (2000), 309--368.

\bibitem{Lcuts} E. Lerman, Contact cuts, \emph{Israel J. of Math.} 124 (2001), 77--92.
\bibitem{Lerman} E. Lerman, Contact toric manifolds, \emph{J. Symplectic Geom}. 1 (2003), 785--828.
\bibitem{Lerman2} E. Lerman, Maximal tori in the contactomorphism groups of circle bundles over Hirzebruch surfaces, \emph{Math. Res. Lett.} 10 (2003), 133--144.
\bibitem{LermanTolman} E. Lerman and S. Tolman, Hamiltonian torus actions on symplectic orbifolds and toric varieties, \emph{Trans. Amer. Math. Soc}. 349 (1997), 4201--4230. 

\bibitem{marinkovic} A. Marinkovi\'c, Symplectic fillability of toric contact manifolds, \emph{Period. Math. Hungar.} 73 (2016), 16--26.

\bibitem{marinkovicpabiniak} A. Marinkovi\'c and M. Pabiniak, Every symplectic toric orbifold is a centered reduction of a Cartesian product of weighted projective spaces, \emph{Int. Math. Res. Not.} 2015 (2015), 12432-12458.

\bibitem{mcduff} D. McDuff, Displacing Lagrangian toric fibers via probes, Low-dimensional and symplectic topology, volume 82 of Proc. Sympos. Pure Math. \emph{Amer. Math. Soc.}, Providence, RI, (2011), 131--160.
\bibitem{Sandon}S. Sandon, Contact homology, capacity and non-squeezing in $\mathbb{R}^{2n}\times\mathbb{S}^1$ via generating
functions, \emph{Ann. Inst. Fourier (Grenoble)}, 61 (2011), 145--185.


\end{thebibliography}
\end{document}